\documentclass[11pt]{amsart}
\usepackage{mathrsfs}
\usepackage{amsmath}
\usepackage{amsfonts}
\usepackage{amssymb}
\usepackage{amsxtra}
\usepackage{mathtools}
\usepackage{dsfont}

\usepackage{graphicx}
\usepackage[T1]{fontenc}
\usepackage{newtxmath}
\usepackage{newtxtext}
\usepackage[margin=2cm, centering]{geometry}
\usepackage{enumitem}

\usepackage[colorlinks,citecolor=blue,urlcolor=blue,bookmarks=true]{hyperref}
\hypersetup{
pdfpagemode=UseNone,
pdfstartview=FitH,
pdfdisplaydoctitle=true,
pdfborder={0 0 0}, 
pdftitle={Variational estimates for operators over some thin subsets of primes},
pdfauthor={Bartosz Trojan},
pdflang=en-US
}

\usepackage[english,polish]{babel}
\newcommand{\pl}[1]{\foreignlanguage{polish}{#1}}

\newtheorem{theorem}{Theorem}
\newtheorem{proposition}{Proposition}[section]
\newtheorem{lemma}{Lemma}

\newcounter{thm}

\newtheorem{main_theorem}[thm]{Theorem}

\theoremstyle{definition}

\newtheorem{definition}{Definition}

\newcommand{\RR}{\mathbb{R}}
\newcommand{\ZZ}{\mathbb{Z}}

\newcommand{\NN}{\mathbb{N}}

\newcommand{\PP}{\mathbb{P}}

\newcommand{\bfP}{\mathbf P}

\newcommand{\scrA}{\mathscr{A}}

\newcommand{\scrH}{\mathscr{H}}

\renewcommand{\atop}[2]{\substack{{#1}\\{#2}}}
\newcommand{\norm}[1]{{\left\lvert #1 \right\rvert}}

\newcommand{\abs}[1]{{\lvert {#1} \rvert}}

\newcommand{\calA}{\mathcal{A}}
\newcommand{\calC}{\mathcal{C}}

\newcommand{\calB}{\mathcal{B}}
\newcommand{\calM}{\mathcal{M}}

\newcommand{\calR}{\mathcal{R}}

\newcommand{\calL}{\mathcal{L}}
\newcommand{\calO}{\mathcal{O}}
\newcommand{\calJ}{\mathcal{J}}
\newcommand{\calH}{\mathcal{H}}

\newcommand{\vphi}{\varphi}

\author{Bartosz Trojan}
\address{
	Bartosz Trojan\\
	Instytut Matematyczny
	Polskiej Akademii Nauk\\
	ul. \pl{{\'S}niadeckich} 8\\
	00-656 Warszawa\\
	Poland}
\email{btrojan@impan.pl}

\title[Variational estimates]
{Variational estimates for operators over some thin subsets of primes}

\begin{document}
\selectlanguage{english}

\begin{abstract}
	We establish $\ell^p(\ZZ)$ boundedness of $r$-variational seminorm for operators of Radon type along subsets
	of prime numbers of the form $\big\{p \in \PP : \{ \varphi_1(p)\} < \psi(p)\big\}$. As an application we obtain the
	corresponding pointwise ergodic theorems.
\end{abstract}

\maketitle

\section{Introduction}
Given a dynamical system $(X, \calB, \mu, T)$, that is a $\sigma$-finite measure space $(X, \calB, \mu)$ with 
an invertible measure preserving transformation $T : X \rightarrow X$, and any polynomial $P: \ZZ \rightarrow \ZZ$ of
degree $d \geq 1$ having integer coefficients and without a constant term, we are interested in the pointwise convergence
for $f \in L^s(X, \mu)$, $s > 1$, of the averages
\[
	\scrA_N f(x) = \frac{1}{\abs{\bfP \cap [1, N]}} \sum_{p \in \bfP \cap [1, N]} f \left(T^{P(p)}x\right)
\]
where $\bfP$ is a thin subset of prime numbers $\PP$, i.e. a subset of $\PP$ such that
\[
	\lim_{N \to \infty} \frac{|\bfP \cap [1, N]|}{|\PP \cap [1, N]|} = 0.
\]
Our principle example is the set
\[
	\bfP = \big\{p \in \PP : p = \lfloor h(n) \rfloor \text{ for some } n \in \NN \big\}
\]
where $h$ is a regularly-varying function of index $c \in [1, 2)$, for instance
$x^c \log^A(x)$ for some $A > 0$. In this context, we also study pointwise convergence of the truncated discrete Hilbert
transform with an appropriate weight function $\omega$,
\[
	\scrH_N f(x) = \sum_{p \in \pm \bfP \cap [1, N]} f\left(T^{P(p)} x\right) \frac{\omega(\abs{p})}{p}.
\]
The problem we are interested in may be stated as follows: for a subset $\mathbf{A} \subseteq \NN$, $s \geq 1$
and any polynomial $P$ having integer coefficients and without a constant term, determine whether for any function
$f \in L^s(X, \mu)$, the limit
\begin{equation}
	\label{eq:1}
	\lim_{N \to \infty} \frac{1}{|\mathbf A \cap [1, N] |} 
	\sum_{n \in \mathbf A \cap [1, N]} f\left(T^{P(n)} x \right)
\end{equation}
exists for $\mu$-almost all $x$. 

Pointwise convergence of ergodic averages was initially observed by Birkhoff in
\cite{birk} where the author considered $\mathbf A = \NN$, $P(n) = n$ and $s \geq 1$. The higher degree polynomials
required a new approach discovered by Bourgain in 80's. In the series of papers, \cite{bou1, bou2, bou}, Bourgain proved
the pointwise convergence for $\mathbf A=\NN$, any polynomial $P$ having integer coefficients, and $s > 1$.
The restriction to the range $s > 1$, in Bourgain's theorem is essential. In fact, Buczolich and Mauldin \cite{BM},
and LaVictoire \cite{LaV} showed that in the case of $P(n) = n^k$, $k \geq 2$, the pointwise convergence \eqref{eq:1}
for a function in $L^1(X,\mu)$ may fail on a large set.

Considering averages over prime numbers, in \cite{bou-p} Bourgain proved their pointwise convergence for $P(n) = n$ and
functions in $L^2(X, \mu)$. Later, in \cite{wrl}, Wierdl extended this result to all $s > 1$,
(see also \cite[Section 9]{bou}). Again the restriction $s > 1$, is essential as LaVictoire showed in \cite{LaV}.
The case of higher degree polynomials, at least for functions in $L^2(X, \mu)$, was investigated by Nair in \cite{na1}.
In \cite{na2}, Nair also studied $s > 1$ but his proof of Lemma 14 contains an error. The general case $s > 1$,
I have covered in the recent paper \cite{tr1}. Finally, a subclass of thin subsets of primes discussed in this article
were previously studied by Mirek in \cite{m2}.

The initial study of pointwise convergence for the truncated discrete Hilbert transform goes back to Cotlar \cite{cot},
where $\mathbf A = \NN$ and $P(n) = n$ was considered. The case with a general polynomial $P$ was a more delicate issue
recently resolved by Mirek, Stein and myself in \cite{mst2}. On the other hand, the truncated discrete Hilbert transform
along prime numbers was the subject of the article written by Mirek and myself \cite{mt2},
see also \cite{mtz}. Ultimately, the general polynomials I have considered in the recent paper \cite{tr1}.

Classical proofs of pointwise convergence proceeds in two steps: The first, is to establish the convergence for a
class of functions dense in $L^p(X, \mu)$. To extend the result to all functions, one needs $L^p$-boundedness of the
corresponding maximal function. Nevertheless, finding the dense class may be a difficult task. To overcome this,
one can show the $r$-variational estimates, see Theorem \ref{thm:1} and Theorem \ref{thm:2} for details.
This approach to study discrete operators has already been used in many papers, see 
\cite{cjrw, jkrw, K, mst2, mt3, mtz, tr1, zk}.

Before stating the results, let us define thin subsets of $\PP$ we are interested in. 
\begin{definition}
	\label{def:1}
	Let $\mathcal L$ be a family of slowly varying functions
	$L: [1, \infty) \rightarrow (0, \infty)$  such that
	\begin{align*}
		L(x)=\exp\Big(\int_1 ^x \frac{\vartheta (t)}{t} {\: \rm d} t \Big)
	\end{align*}
	where $\vartheta\in \calC^{\infty}([1, \infty))$ is a real function satisfying
	\[
		\lim_{x\to\infty}
		\vartheta(x)=0, \qquad\text{and}\qquad
	 	\lim_{x\to\infty}
		x^n\vartheta^{(n)}(x)=0,
		\qquad\text{for every } n\in\NN.
	\]
\end{definition}
Let us distinguish a subfamily $\calL_0$ of $\calL$.
\begin{definition}
	\label{def:2}
	Let $\calL_0$ be a family of  slowly varying functions
	$L:[1, \infty) \rightarrow (0, \infty)$ such that $\lim_{x\to\infty}L(x)=\infty$ and
	\begin{align*}
		L(x)=\exp\Big(\int_1 ^x\frac{\vartheta (t)}{t} {\: \rm d} t\Big)
	\end{align*}
	where  $\vartheta \in \calC^{\infty}([1, \infty))$ is positive decreasing real function satisfying
	\begin{align*}
		\lim_{x\to\infty} \vartheta(x)=0,\qquad\text{and}\qquad
		\lim_{x\to\infty} \frac{x^n\vartheta^{(n)}(x)}{\vartheta(x)}=0,
		\qquad\text{for every } n\in\NN,
	\end{align*}
	and for every $\epsilon>0$ there is a constant $C_{\epsilon}>0$ such that
	$1 \le C_{\epsilon}\vartheta(x)x^{\epsilon}$.
\end{definition}
Lastly, we define the subfamily $\calR_c$ of regularly varying functions.
\begin{definition}
	\label{def:3}
	For every $c \in(0, 2)$, let $\mathcal R_c$ be a family of increasing convex regularly-varying functions
	$h:[1, \infty) \rightarrow [1, \infty)$ of the form
	\[
		h(x)=x^c L(x),
	\]
	where $L \in \mathcal L_0$, if $c = 1$, and $L \in \mathcal L$ otherwise.
\end{definition}
Let us fix two functions $h_1 \in \calR_{c_1}$ and $h_2 \in \calR_{c_2}$ for $c_1, c_2 \in [1, 2)$.
In the whole article it is assumed that $\gamma_1 = 1/c_1$ and $\gamma_2 = 1/c_2$ satisfy
\begin{enumerate}
	\item if $d = 1$,
	\[
		\left\{
		\begin{aligned}
		&(1-\gamma_1) &+ 15 &(1-\gamma_2) &< 1, \\
		3&(1-\gamma_1) &+ 12 &(1-\gamma_2) &< 2,
		\end{aligned}
		\right.
	\]
	\item if $d = 2$,
	\[
		\left\{
		\begin{aligned}
		3&(1-\gamma_1) &+ 62&(1-\gamma_2) &< 3, \\
		4&(1-\gamma_1) &+ 32&(1-\gamma_2) &< 3,
		\end{aligned}
		\right.
	\]
	\item if $d \in \{3, \ldots, 9\}$,
	\[
		\left\{
		\begin{aligned}
		\frac{1}{3 \cdot 2^d} (1-\gamma_1) + \bigg(1+\frac{1}{6(2^d-1)}\bigg) &(1- \gamma_2)
		&< \frac{1}{3 \cdot 2^d}, \\
		&(1-\gamma_2) &< \frac{1}{4\cdot 2^d},
		\end{aligned}
		\right.
	\]
	\item if $d \geq 10$,
	\[
		\frac{2}{3d(d+1)^2} (1 -\gamma_1) + \bigg(1+\frac{1}{3d(d+1)}\bigg)(1-\gamma_2) 
		< \frac{2}{3d(d+1)^2}.
	\]
\end{enumerate}
Let $\varphi_1$ and $\varphi_2$ be the inverse of $h_1$ and $h_2$, respectively. By \cite[Lemma 2.20]{m2}, if $c_j = 1$
then there is a positive real decreasing function $\sigma_j$ satisfying $\sigma_j(2x) \simeq \sigma_j(x)$ and
$\sigma_j(x) \gtrsim x^{-\epsilon}$ for any $\epsilon > 0$, such that for each $k \in \NN$,
\footnote{We write $A \lesssim B$ if there is an absolute constant $C>0$ such that $A \le CB$.
If $A \lesssim B$ and $B \lesssim A$ hold simultaneously then we write $A \simeq B$.} 
\begin{equation}
	\label{eq:28}
	\varphi_j^{(k)} (x) \simeq \frac{\varphi_j(x) \sigma_j(x)}{x^k}.
\end{equation}
We set $\sigma_j \equiv 1$ whenever $c_j > 1$. In this article, we are interested in sets of the form
\[
	\bfP_+=\big\{p \in \PP: \{\varphi_1(p)\}<\psi(p) \big\},
	\qquad \text{and} \qquad
	\bfP_-=\big\{p \in \PP: \{-\varphi_1(p)\}<\psi(p) \big\},
\]
where $\psi: [1, \infty) \rightarrow (0, \infty)$ is a positive function such that 
$\psi(x) \leq \tfrac{1}{2}$ for all $x \geq 1$, and
\begin{align}
	\label{eq:5}
	\lim_{x \to +\infty} \frac{\psi^{(k)}(x)}{\varphi_2^{(k+1)}(x)} = 1,
\end{align}
for $k = 0, \ldots, d+2$, where $d \geq 1$ is the degree of the polynomial $P$. The sets $\bfP_-$ and $\bfP_+$ are
intersections with primes numbers of sets studied in \cite{kmt0}.

Let us observe that, if $h_1=h_2=h$ is the inverse function to $\varphi$ and $\psi(x)=\varphi(x+1)-\varphi(x)$ then
\[
	\bfP_-=\left\{p \in \PP : p = \lfloor h(n)\rfloor \text{ for some } n \in \NN \right\}.
\]
Indeed, we have the following chain of equivalences
\begin{align*}
	\PP \ni p = \lfloor h(n) \rfloor \ \text{ for some } n \in \NN 
	& \iff h(n)-1 < p \leq h(n) < p+1 \\
	& \iff \varphi(p) \leq n < \varphi(p+1), \; \; \text{ since $\varphi$ is increasing} \\
	& \iff 0 \leq n - \varphi(p) < \varphi(p+1) - \varphi(p) = \psi(p) \leq \tfrac{1}{2} \\
	& \iff 0 \leq \{ - \varphi(p) \} < \psi(p) \\
	& \iff p \in \bfP_-.
\end{align*}
In particular, the sets $\bfP_-$ are a generalization of those considered by Leitmann \cite{leit} and Mirek
\cite{m2}. 

For any $r\ge1$, the $r$-variational seminorm $V_r$ of a sequence $\big(a_n: n\in \NN\big)$ of complex numbers is defined
by
\[
	V_r\big(a_n : n\in \NN\big) 
	=
	\sup_{k_0<\ldots <k_J}
	\Big(\sum_{j=1}^J|a_{k_j}-a_{k_{j-1}}|^r\Big)^{1/r}.
\]
Observe that, if $V_r(a_ n : n \in \NN) < \infty$ for any $r \geq 1$, then the sequence $(a_n : n \in \NN)$
convergences. Therefore, we can deduce the pointwise ergodic theorems from the following two statements.
\begin{main_theorem}
\label{thm:1}
	Let $\bfP \in \big\{\bfP_-, \bfP_+\big\}$. For every $s > 1$ there is $C_s > 0$ such that for all $r>2$
	and any $f \in L^s(X, \mu)$,
	\[
		\big\lVert
		V_r\big(  \scrA_N f: N\in\NN\big)
		\big\rVert_{L^s}
		\le
		C_s \frac{r}{r-2} \|f\|_{L^s}.
	\]
	Moreover, the constant $C_s$ is independent of coefficients of the polynomial $P$.
\end{main_theorem}
\begin{main_theorem}
	\label{thm:2}
	Let $\bfP \in \big\{\bfP_-, \bfP_+\big\}$. For every $s>1$ there is $C_s > 0$ such that for all $r>2$ and 
	any $f \in L^s(X, \mu)$,
	\[
		\big\lVert
		V_r\big(\scrH_N f: N\in\NN\big) \big\rVert_{L^s}\le
		C_s \frac{r}{r-2}\|f\|_{L^s}.
	\]
	Moreover, the constant $C_s$ is independent of coefficients of the polynomial $P$.
\end{main_theorem}
We point out that Theorem \ref{thm:2} allows us to define ergodic counterpart of the singular integral operator. Namely,
for $f \in L^s(X, \mu)$, $s > 1$, we set
\[
	\scrH f(x) = \lim_{N \to \infty} \scrH_N f(x),
\]
for $\mu$-almost all $x \in X$. 

In view of the Calder\'on transference principle while proving Theorem \ref{thm:1} and Theorem \ref{thm:2} we may assume
that we deal with the model dynamical system, namely, the integers $\ZZ$ with the counting measure and the shirt operator.
As usual, $r$-variations are divided into to two parts: short and long variations. By choosing long
variations to be over the set $Z_\rho = \big\{\lfloor 2^{k^\rho} \rfloor : k \in \NN \big\}$ for some $\rho \in (0, 1)$,
we make short variations easier to handle. Indeed, bounding short variations is reduced to estimating $\ell^1(\ZZ)$-norm
of convolution kernels, which is a consequence of the asymptotic of some exponential sums over $\bfP$ combined with
the prime number theorem or the Mertens theorem. For long variations, we replace the operators modeled on $\bfP$
by operators modeled on $\PP$. For this step, we need to establish a decay of $\ell^2$-norm of the corresponding
difference. In view of the Plancherel's theorem, it is a consequence of estimates for some exponential sums over $\bfP$,
see Section \ref{sec:1}. Lastly, variational estimates for the operators modeled on $\PP$ are proved in
\cite[Theorem C]{tr1}.

\section{Exponential sums}
\label{sec:1}
In this section we develop estimates on exponential sums that are essential to our argument. The main tools is
van der Corput's lemma in the classical form as well as the one recently obtained by Heath-Brown
(see \cite[Theorem 1]{hb2}).
\begin{lemma}[{\cite{VDC}, \cite[Theorem 5.11, Theorem 5.13]{titch}}]
	\label{lem:4}
	Suppose that $N\ge1$ and $k\geq2$ are two integers and $a\le b \le a+N$. Let $F\in\mathcal{C}^k(a, b)$ be a
	real-valued function such that
	\begin{align*}
  		\eta \lesssim |F^{(k)}(x)|\lesssim r\eta,\ \ \mbox{for all \ $x\in (a, b)$,}
	\end{align*}
	for some $\eta>0$ and $r\ge1$. Then 
	\begin{align*}
		\bigg|\sum_{a\le n\le b}
		e^{2\pi i F(n)}\bigg|\lesssim
    	N\left(\eta^{\frac{1}{2^k-2}}+N^{-\frac{2}{2^k}}+(N^k\eta)^{-\frac{2}{2^k}}\right).
	\end{align*}
	The implied constant depends only on $r$.
\end{lemma}

\begin{lemma}[\cite{hb2}]
	\label{lem:3}
	Suppose that $N \geq 1$ and $k \geq 3$ are two integers and $a \leq b \leq a + N$. Let $F \in \calC^k(a, b)$ be a
	real-valued function such that
	\[
		\eta \lesssim \abs{F^{(k)}(x)} \lesssim r \eta,
		\quad\text{for all } x \in (a, b),
	\]
	for some $\eta > 0$ and $r \geq 1$. Then for every $\epsilon > 0$,
	\[
		\bigg|\sum_{a \leq n \leq b} e^{2\pi i F(n)} \bigg|
		\lesssim
		N^{1+\epsilon}
		\Big(\eta^{\frac{1}{k(k-1)}} + N^{-\frac{1}{k(k-1)}} + (N^k \eta)^{-\frac{2}{k^2(k-1)}}\Big),
	\]
	where the implied constant depends only on $r$, $k$ and $\epsilon$.
\end{lemma}
Notice that the exponents in Lemma \ref{lem:3} are improved for $n \geq 10$. In fact, the second term in the
bracket has smaller exponent in Lemma \ref{lem:4} for $2 \leq n \leq 5$, while the third term for $2 \leq n \leq 9$.
To benefit from this observation, we take the minimum of both estimates.

We start by investigating some exponential sums over integers in arithmetic progression.
\begin{proposition}
	\label{prop:2}
	For $m \in \ZZ \setminus \{0\}$, $\tau \in \{0, 1\}$ and $j \geq 1$, we set
	\[
		T(K) = \sum_{1 \leq k \leq K} \exp\Big(2\pi i \big(\xi P(j k) + m(\varphi_1(j k) - \tau \psi(j k))\big)\Big).
	\]
	Then
	\footnote{We write $A \lesssim_\delta B$ to indicate that the implied constant depends on $\delta$.}
	\begin{enumerate}[leftmargin=*]
	\item if $d \geq 1$ then for each $\epsilon > 0$,
	\[
		\big| T(K) \big| 
		\lesssim_\epsilon 
		\abs{m}^\frac{1}{2(2^d-1)} (j K)^{1+\epsilon} \big(\varphi_1(j K) \sigma_1(j K) \big)^{-\frac{1}{2^d}}, 
	\]
	\item if $d \geq 2$ then for each $\epsilon > 0$,
	\[
		\big| T(K) \big|
		\lesssim_\epsilon \abs{m}^\frac{1}{d(d+1)} (j K)^{1+\epsilon}
		\big(\varphi_1(j K) \sigma_1(j K)\big)^{-\frac{2}{d(d+1)^2}}.
	\]
	\end{enumerate}
	The implied constants are independent of $j$, $m$, $\tau$, $K$ and $\xi$.
\end{proposition}
\begin{proof}
	For the proof, let us define $F: [1, \infty) \rightarrow \RR$ by
	\[
		F(t) = \xi P(j t) + m \big(\varphi_1(j t) - \tau \psi(j t)\big).
	\]
	By \eqref{eq:28} and \eqref{eq:5}, 
	\[
		\psi^{(d+1)}(x) \simeq \varphi_2^{(d+2)}(x) \simeq 
		\frac{\varphi_2(x)\sigma_2(x)}{x^{d+2}},
	\]
	and since $\gamma_2 \leq 1 \leq 1 + \gamma_1$, we have
	\[
		\frac{ \varphi_2(x) \sigma_2(x) }{ x \varphi_1(x) \sigma_1(x) } = o(1),
	\]
	thus
	\[
		\psi^{(d+1)}(x) = o\bigg(\frac{\varphi_1(x) \sigma_1(x)}{x^{d+1}}\bigg).
	\]
	Hence, by \eqref{eq:28}, for $t \in [X, 2X]$, we obtain
	\[
		\big| F^{(d+1)}(t)\big| = j^{d+1} \abs{m} \cdot \big| \varphi_1^{(d+1)}(j t) - \tau \psi^{(d+1)}(j t) \big|
		\simeq
		\frac{j^{d+1} \abs{m} \varphi_1(j X) \sigma_1(j X)}{(j X)^{d+1}}.
	\]
	For $X < X' \leq 2X$, we set
	\[
		T(X, X') = \sum_{X < k \leq X'} e^{2\pi i F(k)}.
	\]
	Then
	\begin{equation}
		\label{eq:53}
		\big| T(K) \big| \lesssim (\log K) \max_{\atop{X < X' \leq K}{X' \leq 2X}} \big|T(X, X')\big|.
	\end{equation}
	Since for each satisfying $\delta < \gamma_1^{-1}$ and $\delta \leq 1$ if $\gamma_1 = 1$, a function
	$x \mapsto x (\varphi_1(x) \sigma_1(x))^{-\delta}$ is increasing, see \cite[Lemma 2.6]{m3}, by Lemma \ref{lem:4} and
	Lemma \ref{lem:3}, we obtain respectively
	\begin{align*}
		\big| T(X, X') \big| &\lesssim 
		X \bigg(\frac{j^{d+1} \abs{m} \varphi_1(j X) \sigma_1(j X)}{(j X)^{d+1}}\bigg)^{\frac{1}{2(2^d-1)}}
		+ X^{1-\frac{1}{2^d}}
		+ X \bigg(X^{d+1} \frac{j^{d+1} \abs{m} \varphi_1(jX) \sigma_1(j X)}{(j X)^{d+1}}\bigg)^{-\frac{1}{2^d}} \\
		&\lesssim
		(\abs{m} j)^{\frac{1}{2(2^d-1)}} 
		X^{1 - \frac{d}{2(2^d - 1)}}
		+ X^{1-\frac{1}{2^d}}
		+ X \big(\abs{m} \varphi_1(j X) \sigma_1(j X) \big)^{-\frac{1}{2^d}} \\
		&\lesssim
		\abs{m}^{\frac{1}{2(2^d -1)}} j X \big(\varphi_1(j X) \sigma_1(j X) \big)^{-\frac{1}{2^d}},
	\end{align*}
	and
	\begin{align*}
		\big| T(X, X') \big| &\lesssim
		X^{1 + \epsilon} 
		\bigg(\frac{j^{d+1} \abs{m} \varphi_1(j X) \sigma_1(j X)}{(j X)^{d+1}}\bigg)^{\frac{1}{d(d+1)}}
		+X^{1 + \epsilon -\frac{1}{d (d+1)}} \\
		&\phantom{\lesssim}
		+X^{1 +\epsilon}\bigg(X^{d+1} \frac{j^{d+1} \abs{m} \varphi_1(j X) \sigma(j X)}{(j X)^{d+1}}\bigg)
		^{-\frac{2}{d (d+1)^2}}\\
		&\lesssim
		(\abs{m} j)^{\frac{1}{d(d+1)}}
		X^{1+\epsilon - \frac{1}{d+1}}
		+ X^{1 + \epsilon - \frac{1}{d(d+1)}}
		+ X^{1+\epsilon} \big(\abs{m} \varphi_1(j X) \sigma_1(j X) \big)^{-\frac{2}{d(d+1)^2}} \\
		&\lesssim
		\abs{m}^{\frac{1}{d(d+1)}} (j X)^{1+\epsilon} \big(\varphi_1(j X) \sigma_1(j X)\big)^{-\frac{2}{d(d+1)^2}}.
	\end{align*}
	Now, using \eqref{eq:53} we easily finish the proof.
\end{proof}
Let us turn to estimating the exponential sums over prime numbers. To regularize them we use von Mangoldt's
function defined as
\[
	\Lambda(n) =
	\begin{cases}
		\log p & \text{if } n = p^m, \text{ for some } p \in \PP \text{ and } m \in \NN, \\
		0 & \text{otherwise}.
	\end{cases}
\]
The classical way to handle von Mangoldt's function is to use Vaughan's identity (see \cite{vau0}, see also 
\cite[Lemma 4.12]{GK}), which states that for any $n > u \geq 1$,
\begin{equation}
	\label{eq:36}
	\Lambda(n) = 
	\sum_{\atop{j, k > u}{j k = n}} \Lambda(k) a_j
	+ \sum_{\atop{j \leq u}{j k = n}} \mu(j) \log(k)
	- \sum_{\atop{j \leq u^2}{j k = n}} b_j,
\end{equation}
where 
\[
	a_j = \sum_{\atop{d > u}{d \ell = j}}\mu(d),\qquad
	b_j = \sum_{\atop{d, \ell \leq u}{d \ell = j}} \mu(d) \Lambda(\ell),
\]
and $\mu(n)$ is the M\"obius function defined for $n = p_1^{m_1} \cdots p_k^{m_k}$, where $p_j$ are distinct prime numbers,
as
\[
	\mu(n) = 
	\begin{cases}
		(-1)^k  & \text{if } m_1 = \ldots = m_k, \\
		0 & \text{otherwise.}
	\end{cases}
\]
Let us observe that for any $\epsilon > 0$,
\[
	\sum_{J \leq j \leq 2J} \abs{a_j}^2 \lesssim_{\epsilon} J^{1 + \epsilon},
	\qquad\text{and}\qquad
	\sum_{J \leq j \leq 2J} \abs{b_j}^2 \lesssim_{\epsilon} J^{1 + \epsilon}.
\]
\begin{theorem}
	\label{thm:5}
	For $m \in \ZZ \setminus \{0\}$, $\tau \in \{0, 1\}$ and $1 \leq X \leq X' \leq 2X$, we set
	\[
		S(X, X') = \sum_{X < n \leq X'} 
		\exp\Big(2\pi i \big(\xi P(n) + m(\varphi_1(n) - \tau \psi(n))\big)\Big) \Lambda(n).
	\]
	Then for each $\epsilon > 0$,
	\begin{enumerate}[itemindent=0pt,leftmargin=*]
	\item if $d = 1$,
	\[
		\big| S(X, X') \big|
        \lesssim_\epsilon
		X^{1+\epsilon}
		\Big(
		\abs{m}^{\frac{1}{4}} X^{-\frac{1}{12}}
		+ \abs{m}^{\frac{1}{14}} \big(\varphi_1(X) \sigma_1(X)\big)^{-\frac{1}{14}}
		+ X^{\frac{1}{12}} \big(\varphi_1(X) \sigma_1(X)\big)^{-\frac{1}{4}}
		\Big),
	\]
	\item if $d = 2$,
	\[
		\big| S(X, X') \big|
        \lesssim_\epsilon
        X^{1+\epsilon}
        \Big(
        \abs{m}^{\frac{1}{12}} X^{-\frac{1}{16}}
		+
		\abs{m}^{\frac{1}{30}}
		\big(\varphi_1(X) \sigma_1(X)\big)^{-\frac{1}{20}}
		+
		X^{\frac{1}{32}} \big(\varphi_1(X) \sigma_1(X)\big)^{-\frac{1}{8}}
		\Big),
	\]
	\item if $d \in \{3, \ldots, 9\}$,
	\[
		\big| S(X, X') \big|
        \lesssim_\epsilon
        X^{1+\epsilon}
        \Big(
		X^{-\frac{1}{4\cdot 2^d}}
		+\abs{m}^{\frac{1}{4(2^d-1)}} X^{-\frac{d-1}{8(2^d-1)}}
		+\abs{m}^{\frac{1}{6(2^d-1)}} 
		\big(\varphi_1(X) \sigma_1(X)\big)^{-\frac{1}{3\cdot2^d}}
		\Big),
	\]
	\item if $d \geq 10$,
	\[
		\big| S(X, X') \big| 
		\lesssim_\epsilon
		X^{1+\epsilon} 
		\Big(
		X^{-\frac{1}{4d(d+1)}} 
		+ \abs{m}^{\frac{1}{2d(d+1)}} X^{-\frac{d-1}{4d(d+1)}}
		+ \abs{m}^{\frac{1}{3d(d+1)}} \big(\varphi_1(X) \sigma_1(X)\big)^{-\frac{2}{3 d(d+1)^2}}
		\Big). 
	\]
	\end{enumerate}
	The implied constants are independent of $m$, $\tau$, $X$, $X'$ and $\xi$.
\end{theorem}
\begin{proof}
	To simplify notation, let $F: [1, \infty) \rightarrow \RR$ stand for 
	\[
		F(t) = \xi P(t) + m(\varphi_1(t) - \tau \psi(t)). 
	\]
	Fix $1 \leq u \leq X^\frac{1}{3}$ whose value will be determined later. By Vaughan's identity \eqref{eq:36}, 
	we can write
	\[
		S(X, X') = \Sigma_1 - \Sigma_{21} - \Sigma_{22} + \Sigma_3,
	\]
	where
	\begin{align*}
		\Sigma_1 &= \sum_{j \leq u} \mu(j) \sum_{X/j < k \leq X'/j} e^{2\pi i F(j k)} \log(k), \\
		\Sigma_{21} &= \sum_{j \leq u} b_j \sum_{X/j < k \leq X'/j} e^{2\pi i F(j k)},  \\
		\Sigma_{22} &= \sum_{u < j \leq u^2} b_j \sum_{X/j < k \leq X'/j} e^{2\pi i F(j k)}, \\
		\Sigma_3 &=
		\sum_{u < j \leq X'/u} a_j \sum_{\atop{X/j < k \leq X'/j}{k > u}} e^{2\pi i F(j k)} \Lambda(k).
	\end{align*}
	Therefore, our aim is reduced to bounding each term separately.

	\noindent
    {\bf The estimate for $\Sigma_1$ and $\Sigma_{21}$.} 
	For $1 \leq j \le u$ we set
	\[
		T_j(K) = \sum_{X/j < k \leq K} e^{2 \pi i F(j k)}.
	\]
	By the partial summation, we can write
	\[
		\sum_{X/j < k \leq X'/j} e^{2\pi i F(j k)} \log(k)
		=
		T_j(X'/j) \log (X'/j)
		- \int_{X}^{X'} T_j(t/j) \frac{{\rm d} t}{t},
	\]
	thus
	\begin{align*}
		\big|\Sigma_1\big| \lesssim
		(\log X) \sum_{j \leq u} \max_{X/j \leq K \leq X'/j} \big| T_j(K) \big|.
	\end{align*}
	Moreover, since
	\[
		\abs{b_j} \leq \sum_{\ell \mid j} \Lambda(\ell) = \log (j),
	\]
	we have
	\[
		\big| \Sigma_{21} \big|
		\lesssim
		(\log X) \sum_{j \leq u} \max_{X/j \leq K \leq X'/j} \big| T_j(K) \big|.
	\]
	Therefore, by Proposition \ref{prop:2}(i), we obtain
	\begin{align}
		\nonumber
		\abs{\Sigma_1}+\abs{\Sigma_{21}} 
		&\lesssim
		(\log X) \sum_{j \leq u} \max_{X/j \leq K \leq X'/j} 
		\abs{m}^{\frac{1}{2(2^d-1)}} (j K)^{1+\epsilon} \big(\varphi_1(jK) \sigma_1(j K) \big)^{-\frac{1}{2^d}}\\
		\label{eq:51}
		&\lesssim
		u
		\abs{m}^{\frac{1}{2(2^d-1)}} 
		X^{1+2 \epsilon} \big(\varphi_1(X) \sigma_1(X) \big)^{-\frac{1}{2^d}}.
	\end{align}
	Similarly, Proposition \ref{prop:2}(ii) gives
	\begin{align}
		\label{eq:52}
		\abs{\Sigma_1}+\abs{\Sigma_{21}}
        &\lesssim
        u \abs{m}^{\frac{1}{d(d+1)}}
        X^{1+2 \epsilon} \big(\varphi_1(X) \sigma_1(X) \big)^{-\frac{2}{d(d+1)^2}}.
	\end{align}

	\noindent
    {\bf The estimate for $\Sigma_{22}$ and $\Sigma_{3}$.}
	Controlling $\Sigma_{22}$ and $\Sigma_3$ requires more work. First, let us dyadically split the defining sums to get
	\begin{equation}
		\label{eq:45}
		\abs{\Sigma_{22}} \lesssim
		(\log X)^2
		\max_{\atop{u \leq J < J' \leq 2J}{J' \leq u^2}} 
		\max_{\atop{X/u^2 \leq K < K' \leq 2K}{K' \leq X'/u}}
        \bigg|
        \underset{X < j k \leq X'}{
        \sum_{J < j \leq J'} \sum_{K < k \leq K'}}
		e^{2\pi i F(j k)} b_j
        \bigg|,
	\end{equation}
	and
	\begin{equation}
		\label{eq:46}
		\abs{\Sigma_3} \lesssim (\log X)^2 
		\max_{\atop{u \leq J < J' \leq 2J}{J' \leq u^2}} 
		\max_{\atop{X/u^2 \leq K < K' \leq 2K}{K' \leq X'/u}}
		\bigg|
		\underset{X < j k \leq X'}{
		\sum_{J < j \leq J'} \sum_{K < k \leq K'}}
		e^{2 \pi i F(j k)} \Lambda(k) a_j
		\bigg|.
	\end{equation}
	To be able to deal with both cases simultaneously, let us consider two sequences of complex numbers 
	$(A_j : j \in \NN)$ and $(B_k : k \in \NN)$, such that for each $\epsilon > 0$,
	\begin{equation}
		\label{eq:38}
        \sum_{J \leq j \leq 2J} \abs{A_j}^2 \lesssim_\epsilon J^{1+\epsilon},\qquad\text{and}\qquad
        \sum_{K \leq k \leq 2K} \abs{B_k}^2 \lesssim_\epsilon K^{1+\epsilon},
    \end{equation}
	and study exponential sums of a form
	\[
        \underset{X < j k \leq X'}{
        \sum_{J < j \leq J'} \sum_{K < k \leq K'}}
		e^{2 \pi i F(j k)} A_j B_k,
	\]
	where $J < J' \leq 2J$ and $K < K' \leq 2K$. Without loss of generality we may assume that $K \leq J$. By
	Cauchy--Schwarz inequality and \eqref{eq:38}, we have
	\begin{align*}
		\Big|
		\underset{X < j k \leq X'}{
        \sum_{J < j \leq J'} \sum_{K < k \leq K'}}
		e^{2 \pi i F(j k)} A_j B_k
		\Big|^2
		&\lesssim
		J^{1 + \epsilon}
		\sum_{J < j \leq J'}
        \Big|
        \sum_{\atop{K < k \leq K'}{X < j k \leq X'}} e^{2\pi i F(j k)} B_k
        \Big|^2.
	\end{align*}
	To estimate the right-hand side, we expand the square and rearrange terms to get
	\begin{align}
		\label{eq:37}
        \Big|
        \sum_{\atop{K < k \leq K'}{X < j k \leq X'}} e^{2\pi i F(j k)} B_k
        \Big|^2
		&=
		\sum_{\abs{r} \leq K}
		\sum_{\atop{K < k, k+r \leq K'}{X < j k, j(k+r)\leq X'}}
		\exp\Big(2\pi i \big(F(j k) - F(j (k+r))\big)\Big)
        B_k \overline{B_{k+r}}.
	\end{align}
	Therefore,
	\begin{equation}
		\label{eq:35}
		\Big|
		\underset{X < j k \leq X'}{
        \sum_{J < j \leq J'} \sum_{K < k \leq K'}}
		e^{2 \pi i F(j k)} A_j B_k
		\Big|^2
		\lesssim
		J^{1+\epsilon}
        \sum_{\abs{r} \leq K}
		\sum_{K < k, k+r \leq K'}
		\abs{B_k} \abs{B_{k+r}}
		\abs{U_{k, k+r}},
	\end{equation}
	where for $K < k, k' < K'$, we have set
	\[
		U_{k, k'}
		=
		\sum_{j \in \calJ_{k, k'}}
		\exp\Big(2\pi i \big(F(j k) - F(j k')\big)\Big),
	\]
	and $\calJ_{k, k'} = \left(\max\big\{X/k, X/k', J\big\}, \min\big\{X'/k, X'/k', J'\big\}\right] \cap \ZZ$. To estimate
	$U_{k, k'}$, we are going to apply van der Corput's lemma. Let us fix $k \neq k'$. Setting
	$G(t) = F(t k) - F(t k')$ for $t \in \calJ_{k, k'}$, we can write
	\[
		\big|G^{(d+1)}(t) \big| \simeq 
		\abs{m} 
		\Big|
		\big(\varphi_1^{(d+1)}(t k) k^{d+1} - \varphi_1^{(d+1)}(t k') (k')^{d+1}\big)
		-
		\tau\big(\psi^{(d+1)}(t k) k^{d+1} - \psi^{(d+1)}(t k') (k')^{d+1}\big)
		\Big|.
	\]
	By the mean value theorem, for some $x$ between $tk$ and $tk'$ we have 
	\[
		\varphi_1^{(d+1)}(t k) (t k)^{d+1} - \varphi_1^{(d+1)}(t k') (t k')^{d+1}
		=
		\big(
		\varphi_1^{(d+2)}(x) x^{d+1} + (d+1) \varphi_1^{(d+1)}(x) x^d
		\big)
		(k - k') t,
	\]
	thus, by \eqref{eq:28}, we obtain
	\[
		\big|
        \varphi_1^{(d+1)}(t k) (t k)^{d+1} - \varphi_1^{(d+1)}(t k') (t k')^{d+1}
        \big|
		\simeq
		\frac{\varphi_1(J K) \sigma_1(J K)}{J K} |k-k'| J.
	\]
	Similarly, we get
	\[
		\big|
        \psi^{(d+1)}(t k) (t k)^{d+1} - \psi^{(d+1)}(t k') (t k')^{d+1}
        \big|
        \simeq
        \frac{\varphi_2(J K) \sigma_2(J K)}{(J K)^2} |k-k'| J.
	\]
	Since
	\[
		\frac{\varphi_2(JK) \sigma_2(JK)}{JK \varphi_1(JK) \sigma_1(JK)} = o(1),
	\]
	we conclude that for $t \in \calJ_{k, k'}$,
	\[
		\big|G^{(d+1)}(t) \big| 
		\simeq
		\abs{m} \cdot |k-k'| \frac{\varphi_1(J K) \sigma_1(J K)}{J K} J^{-d}.
	\]
	Now, by Lemma \ref{lem:4}, we get
	\begin{align}
		\nonumber
		\big|
		U_{k, k'}
		\big|
		&\lesssim
		J
		\bigg(\abs{m} \cdot \abs{k-k'} \frac{\varphi_1(J K) \sigma_1(J K)}{J K} J^{-d} \bigg)^{\frac{1}{2(2^d-1)}}
		+
		J^{1 - \frac{1}{2^d}}\\
		\nonumber
		&\phantom{\lesssim}+
		J
		\bigg(
		J^{d+1}
		\abs{m} \cdot |k-k'| \frac{\varphi_1(J K) \sigma_1(J K)}{J K} J^{-d}
		\bigg)^{-\frac{1}{2^d}}\\
		\label{eq:12}
		&\begin{aligned}
		&\lesssim
		J^{1 - \frac{d}{2(2^d-1)}} \big(\abs{m} \cdot \abs{k-k'}\big)^{\frac{1}{2(2^d-1)}}
		+
		J^{1-\frac{1}{2^d}}\\
		&\phantom{\lesssim}+
		J K^{\frac{1}{2^d}} \big(\abs{m} \cdot \abs{k-k'} \big)^{-\frac{1}{2^d}} 
		\big(\varphi_1(JK) \sigma_1(JK) \big)^{-\frac{1}{2^d}}.
		\end{aligned}
	\end{align}
	By Cauchy--Schwarz inequality and \eqref{eq:38}, we obtain
	\begin{align*}
		&\sum_{1 \leq \abs{r} \leq K} \sum_{K < k, k+r \leq K'} \abs{B_k} \abs{B_{k+r}} 
		J^{1 - \frac{d}{2(2^d-1)}} \abs{m r}^{\frac{1}{2(2^d-1)}} \\
		&\qquad\qquad\lesssim
		J^{1 - \frac{d}{2(2^d-1)}} \abs{m}^{\frac{1}{2(2^d-1)}}
		\sum_{1 \leq \abs{r} \leq K} \abs{r}^{\frac{1}{2(2^d-1)}} \sum_{K < k \leq K'} \abs{B_k}^2 \\
		&\qquad\qquad\lesssim
		J^{1 - \frac{d}{2(2^d-1)}} \abs{m}^\frac{1}{2(2^d-1)}
		K^{1 + \frac{1}{2(2^d-1)}} K^{1+\epsilon}.
	\end{align*}
	Analogously, we show that
	\[
		\sum_{1 \leq \abs{r} \leq K} \sum_{K < k, k+r \leq K'} \abs{B_k} \abs{B_{k+r}}
        J^{1-\frac{1}{2^d}}\\
		\lesssim
		J^{1 - \frac{1}{2^d}} K^{2 + \epsilon},
	\]
	and
	\begin{align*}
		&\sum_{1 \leq \abs{r} \leq K} \sum_{K < k, k+r \leq K'} \abs{B_k} \abs{B_{k+r}} 
		J K^{\frac{1}{2^d}} \abs{m r}^{-\frac{1}{2^d}}
        \big(\varphi_1(JK) \sigma_1(JK) \big)^{-\frac{1}{2^d}} \\
		&\qquad\qquad\lesssim
		J \abs{m}^{-\frac{1}{2^d}} K^{2+\epsilon}
		\big(\varphi_1(JK) \sigma_1(JK) \big)^{-\frac{1}{2^d}}. 
	\end{align*}
	Therefore,
	\begin{align*}
		&
		\sum_{1 \leq \abs{r} \leq K} \sum_{K < k, k+r \leq K'} 
		\abs{B_k} \abs{B_{k+r}} \abs{U_{k, k'}} \\
		&\qquad\qquad\lesssim
		J K^{2 + \epsilon}
		\Big(\abs{m}^{\frac{1}{2(2^d-1)}} J^{-\frac{d}{2(2^d-1)}} K^{\frac{1}{2(2^d-1)}} 
		+ J^{-\frac{1}{2^d}} 
		+ \abs{m}^{-\frac{1}{2^d}} \big(\varphi_1(JK) \sigma_1(JK) \big)^{-\frac{1}{2^d}}\Big).
	\end{align*}
	Since for $r=0$, we have
	\[
		\sum_{K < k \leq K'} \abs{B_k}^2 \abs{U_{k, k}} \lesssim  J K^{1+\epsilon},
	\]
	by \eqref{eq:35}, we can estimate
	\begin{equation}
		\label{eq:39}
		\begin{aligned}
		&\Big|
        \underset{X < j k \leq X'}{
        \sum_{J < j \leq J'} \sum_{K < k \leq K'}}
        e^{2 \pi i F(j k)} A_j B_k
        \Big|^2\\
        &\qquad\qquad\lesssim
		J^{2+\epsilon} K^{2 + \epsilon}
		\Big(K^{-1}
		+
		\abs{m}^{\frac{1}{2(2^d-1)}} J^{-\frac{d}{2(2^d-1)}} K^{\frac{1}{2(2^d-1)}}
		+
		J^{-\frac{1}{2^d}}
		+
		\big(\varphi_1(J K) \sigma_1(J K)\big)^{-\frac{1}{2^d}} \Big),
		\end{aligned}
	\end{equation}
	provided that $K \leq J$. We are now going to apply \eqref{eq:39} to derive the estimates for $\Sigma_{22}$ and
	$\Sigma_3$. Let us recall that $u \leq J \leq u^2$, $X/u^2 \leq K \leq X'/u$, $u^3 < X$ and $X < J K \leq 2X$, thus
	\[
		u \leq \min\{J, K\} \leq \sqrt{3 X} \leq \sqrt{3} \max\{J, K\}.
	\]
	Hence, \eqref{eq:39} applied to \eqref{eq:45} and \eqref{eq:46} results in
	\begin{equation}
		\label{eq:47}
		\begin{aligned}
		\abs{\Sigma_{22}}+ \abs{\Sigma_3}
		\lesssim
		X^{1+\epsilon}
		\Big(
		u^{-\frac{1}{2}}
		+\abs{m}^{\frac{1}{4(2^d-1)}} X^{-\frac{d-1}{8(2^d-1)}}
		+
		X^{-\frac{1}{2^{d+2}}} + \big(\varphi_1(X) \sigma_1(X)\big)^{-\frac{1}{2^{d+1}}}
		\Big).
		\end{aligned}
	\end{equation}
	For $d \in \{1, 2\}$, we improve the estimate \eqref{eq:47}, by applying to \eqref{eq:37} the Weyl--van der Corput's 
	inequality, see \cite[Lemma 2.5]{GK}. For each $1 \leq R \leq K$, we have
	\[
		\sum_{J < j \leq J'}
		\Big|
        \sum_{\atop{K < k \leq K'}{X < j k \leq X'}} e^{2\pi i F(j k)} B_k
        \Big|^2
		\leq
		\bigg(1 + \frac{K}{R}\bigg)
		\sum_{\abs{r} \leq R}
		\bigg(1 - \frac{\abs{r}}{R}\bigg)
		\sum_{K \leq k, k+r \leq K'} \abs{B_k} \abs{B_{k+r}} \abs{U_{k, k'}}.
	\]
	For $d = 1$, we take $R = K^{\frac{1}{3}}$. Then, by \eqref{eq:12}, we get
	\begin{align*}
		\sum_{J < j \leq J'}
		\Big|
        \sum_{\atop{K < k \leq K'}{X < j k \leq X'}} e^{2\pi i F(j k)} B_k
        \Big|^2
		&\lesssim
		J
		K^{1+\epsilon}
		(K + R)
		\Big(R^{-1} 
		+
		\abs{m}^{\frac{1}{2}} J^{-\frac{1}{2}} R^{\frac{1}{2}}
		+
		K^{\frac{1}{2}} R^{-\frac{1}{2}}
		\big(\varphi_1(J K) \sigma_1(J K)\big)^{-\frac{1}{2}}
		\Big) \\
		&\lesssim
		J K^{2+\epsilon}
		\Big(
		K^{-\frac{1}{3}} + \abs{m}^{\frac{1}{2}} J^{-\frac{1}{2}} K^{\frac{1}{6}}
		+ K^{\frac{1}{3}} \big(\varphi_1(J K) \sigma_1(J K)\big)^{-\frac{1}{2}}
		\Big).
	\end{align*}
	Therefore,
	\begin{equation}
		\label{eq:14}
		\abs{\Sigma_{22}} + \abs{\Sigma_3}
        \lesssim
        X^{1+\epsilon}
		\Big(
		u^{-\frac{1}{6}} + \abs{m}^{\frac{1}{4}} X^{-\frac{1}{12}} + X^{\frac{1}{12}} 
		\big(\varphi_1(X) \sigma_1(X)\big)^{-\frac{1}{4}}
		\Big).
	\end{equation}
	Similarly, for $d = 2$; we set $R = K^{\frac{1}{2}}$ which entails that
	\begin{align*}
		&
		\sum_{J < j \leq J'}
		\Big|
		\sum_{\atop{K < k \leq K'}{X < j k \leq X'}} e^{2\pi i F(j k)} B_k
        \Big|^2 \\
		&\qquad\qquad\lesssim
		J
		K^{1+\epsilon}
		(K + R)
		\Big(R^{-1} 
		+
		\abs{m}^{\frac{1}{6}} J^{-\frac{1}{3}} R^{\frac{1}{6}}
		+
		J^{-\frac{1}{4}}
		+
		\abs{m}^{-\frac{1}{4}}
		K^{\frac{1}{4}} R^{-\frac{1}{4}}
		\big(\varphi_1(J K) \sigma_1(J K)\big)^{-\frac{1}{4}}
		\Big) \\
		&\qquad\qquad\lesssim
		J K^{2+\epsilon}
		\Big(
		K^{-\frac{1}{2}} + \abs{m}^{\frac{1}{6}} J^{-\frac{1}{3}} K^{\frac{1}{12}}
		+ \abs{m}^{-\frac{1}{4}} K^{\frac{1}{8}} \big(\varphi_1(J K) \sigma_1(J K)\big)^{-\frac{1}{4}}
		\Big),
	\end{align*}
	and hence
	\begin{equation}
		\label{eq:20}
		\abs{\Sigma_{22}} + \abs{\Sigma_3}
        \lesssim
        X^{1+\epsilon}
		\Big(
		u^{-\frac{1}{4}} + \abs{m}^{\frac{1}{12}} X^{-\frac{1}{16}}
		+
		X^{\frac{1}{32}} 
		\big(\varphi_1(X) \sigma_1(X)\big)^{-\frac{1}{8}}
		\Big).
	\end{equation}
	Next, let us observe that for $d \geq 2$, while estimating $U_{k, k'}$, instead of Lemma \ref{lem:4} we can use
	Lemma \ref{lem:3}. This leads to
	\begin{align*}
		\big|
        U_{k, k'}
        \big|
        &\lesssim
        J^{1+\epsilon - \frac{1}{d+1}}
		\big(\abs{m} \cdot \abs{k-k'} \big)^{\frac{1}{d(d+1)}}
		+
		J^{1 + \epsilon - \frac{1}{d(d+1)}} \\
		&\phantom{\lesssim}+
		J^{1 + \epsilon} K^{\frac{2}{d(d+1)^2}} \big(\abs{m} \cdot \abs{k-k'}\big)^{-\frac{2}{d(d+1)^2}}
		\big(\varphi_1(J K) \sigma_1(J K)\big)^{-\frac{2}{d(d+1)^2}},
	\end{align*}
	and
	\begin{align*}
		&
		\sum_{J < j \leq J'}
		\Big|
		\sum_{\atop{K < k \leq K'}{X < j k \leq X'}}
        e^{2 \pi i F(j k)} B_k
        \Big|^2 \\
		&\qquad\qquad\lesssim
        J^{1+\epsilon} K^{2 + \epsilon} 
		\Big(
		K^{-1} 
		+\abs{m}^{\frac{1}{d(d+1)}} J^{-\frac{1}{d+1}} K^{\frac{1}{d(d+1)}}
		+J^{-\frac{1}{d(d+1)}} 
		+ \big(\varphi_1(J K) \sigma_1(J K)\big)^{-\frac{2}{d(d+1)^2}}
		\Big),
	\end{align*}
	which entails that
	\begin{equation}
		\label{eq:50}
		\abs{\Sigma_{22}} + \abs{\Sigma_3}
        \lesssim
		X^{1+\epsilon}
		\Big(u^{-\frac{1}{2}} 
		+\abs{m}^{\frac{1}{2d(d+1)}} X^{-\frac{d-1}{4d(d+1)}}
		+X^{-\frac{1}{4d(d+1)}} + \big(\varphi_1(X) \sigma_1(X)\big)^{-\frac{1}{d(d+1)^2}}
		\Big).
	\end{equation}

	\noindent
    {\bf Conclusion.}
	In view of the estimates \eqref{eq:51} and \eqref{eq:47}, by selecting
	\[
		u = \abs{m}^{-\frac{2}{6(2^d-1)}} \big(\varphi_1(X) \sigma_1(X)\big)^{\frac{2}{3} \cdot \frac{1}{2^d}},
	\]
	we obtain
	\begin{align*}
		&\big|S(X, X')\big| \\
		&\qquad\lesssim
		X^{1+\epsilon} 
		\Big(u \abs{m}^{\frac{1}{2(2^d-1)}} \big(\varphi_1(X) \sigma_1(X)\big)^{-\frac{1}{2^d}} 
		+ u^{-\frac{1}{2}}
		+
		\abs{m}^{\frac{1}{4(2^d-1)}} X^{-\frac{d-1}{8(2^d-1)}} + X^{-\frac{1}{2^{d+2}}} +
		\big(\varphi_1(X) \sigma_1(X)\big)^{-\frac{1}{2^{d+1}}}
		\Big) \\
		&\qquad\lesssim
		X^{1+\epsilon}
		\Big(
		\abs{m}^{\frac{1}{6(2^d-1)}} \big(\varphi_1(X) \sigma_1(X)\big)^{-\frac{1}{3 \cdot 2^d}}
		+
		\abs{m}^{\frac{1}{4(2^d-1)}} X^{-\frac{d-1}{8(2^d-1)}} + X^{-\frac{1}{2^{d+2}}}
		\Big).
	\end{align*}
	Analogously, setting
	\[
		u = \abs{m}^{-\frac{2}{3 d(d+1)}} \big(\varphi_1(X) \sigma_1(X)\big)^{\frac{4}{3 d (d+1)^2}},
	\]
	from \eqref{eq:52} and \eqref{eq:50}, we get
	\begin{align*}
		\big|S(X, X')\big|
		\lesssim
		X^{1+\epsilon}
		\Big(
		\abs{m}^{\frac{1}{3d(d+1)}} \big(\varphi_1(X) \sigma_1(X)\big)^{-\frac{2}{3 d(d+1)^2}}
		+\abs{m}^{\frac{1}{2d(d+1)}} X^{-\frac{d-1}{4d(d+1)}}
		+X^{-\frac{1}{4d(d+1)}} 
		\Big).
	\end{align*}
	For $d = 1$, we take
	\[
		u = \abs{m}^{-\frac{3}{7}} \big(\varphi_1(X) \sigma_1(X)\big)^{\frac{3}{7}},
	\]
	and use \eqref{eq:51} together with \eqref{eq:14}, to get
	\[
		\big|S(X, X')\big|
        \lesssim
        X^{1+\epsilon}
        \Big(
		\abs{m}^{\frac{1}{14}} \big(\varphi_1(X) \sigma_1(X)\big)^{-\frac{1}{14}}
		+
		\abs{m}^{\frac{1}{4}} X^{-\frac{1}{12}} 
		+
		X^{\frac{1}{12}} \big(\varphi_1(X) \sigma_1(X)\big)^{-\frac{1}{4}}
		\Big).
	\]
	Lastly, for $d = 2$ and
	\[
		u = \abs{m}^{-\frac{2}{15}} \big(\varphi_1(X) \sigma_1(X)\big)^{\frac{1}{5}},
	\]
	by \eqref{eq:51} and \eqref{eq:20}, we obtain
	\[
		\big|S(X, X')\big|
        \lesssim
        X^{1+\epsilon}
        \Big(
        \abs{m}^{\frac{1}{30}} \big(\varphi_1(X) \sigma_1(X)\big)^{-\frac{1}{20}}
        +
        \abs{m}^{\frac{1}{12}} X^{-\frac{1}{16}}
		+
		X^{\frac{1}{32}} \big(\varphi_1(X) \sigma_1(X)\big)^{-\frac{1}{8}}
        \Big),
	\]
	which concludes the proof of theorem.
\end{proof}
The reasoning for $\bfP_+$ and $\bfP_-$ are similar, therefore to simplify the notation we are going to write
\[
	\bfP = \bfP_+ = \big\{p \in \PP : \{\vphi_1(p)\} < \psi(p) \big\}.
\]
For $N \in \NN$ we set $\PP_N = \PP \cap [1, N]$ and $\mathbf{P}_N = \mathbf{P} \cap [1, N]$. In what follows, we need
a characterization of the sets $\bfP$. The proof follows a line parallel to \cite[Lemma 2.2]{kmt0}.
\begin{lemma}
	\label{lem:1}
	$p\in\mathbf P$ if and only if $p \in \PP$ and
	$\lfloor \varphi_1(p) \rfloor - \lfloor \varphi_1(p) - \psi(p) \rfloor = 1$.
\end{lemma}
\begin{proof}
	We begin with the forward implication; it suffices to show that if $p \in \bfP$, then
	the integer
	\[
		\lfloor \varphi_1(p) \rfloor - \lfloor \varphi_1(p) - \psi(p) \rfloor,
	\]
	belongs to $\left(0,\tfrac{3}{2}\right)$. By definition, if
	$p \in \mathbf P$ then $0\leq \varphi_1(p) - \lfloor \varphi_1(p) \rfloor < \psi(p)$, thus
	\[
		- \varphi_1(p) \leq - \lfloor  \varphi_1(p) \rfloor < \psi(p) - \varphi_1(p),
	\]
	if and only if
	\[
		\varphi_1(p) \geq \lfloor  \varphi_1(p) \rfloor > \varphi_1(p) - \psi(p),
	\]
	from where it follows that
	\[
		\lfloor \varphi_1(p)\rfloor - \lfloor \varphi_1(p) - \psi(p) \rfloor
		> \{ \varphi_1(p) - \psi(p) \} \geq 0.
	\]
	In view of $\lfloor  \varphi_1(p) - \psi(p) \rfloor \geq \varphi_1(p) - \psi(p) -1$, we obtain
	\begin{align*}
		\lfloor \varphi_1(p) \rfloor - \lfloor \varphi_1(p) - \psi(p) \rfloor
		& \leq \lfloor  \varphi_1(p) \rfloor - \varphi_1(p) + \psi(p) + 1\\
		& \leq \psi(p) +1 < \tfrac{3}{2}.
	\end{align*}
	We now turn to the reverse implication; if $p \in \PP$ and
	$\lfloor \varphi_1(p) \rfloor = 1 +\lfloor  \varphi_1(p) - \psi(p) \rfloor$, then we have
	\begin{align*}
		0 & \leq \varphi_1(p) - \lfloor  \varphi_1(p) \rfloor
		= \varphi_1(p) -1 - \lfloor  \varphi_1(p) - \psi(p) \rfloor \\
		& < \varphi_1(p) - 1 + 1 + \psi(p) - \varphi_1(p) = \psi(p).
	\end{align*}
	Consequently, we get $\{ \varphi_1(p) \} < \psi(p)$, as desired.
\end{proof}
We are now ready to prove the main theorem of this section. 
\begin{theorem}
	\label{thm:3}
	For each $\epsilon > 0$, satisfying
	\begin{enumerate}[itemindent=*,leftmargin=0pt]
		\item if $d = 1$,
		\[
			\left\{
			\begin{aligned}
			&(1-\gamma_1) &+ 15&(1-\gamma_2) &+ 84 &\epsilon &< 1, \\
			3 &(1-\gamma_1) &+ 12&(1-\gamma_2) &+ 60 &\epsilon &< 2
			\end{aligned}
			\right.
		\]
		\item if $d = 2$,
		\[
			\left\{
			\begin{aligned}
			3&(1-\gamma_1) &+ 62&(1-\gamma_2) &+ 360 &\epsilon &< 3, \\
			4&(1-\gamma_1) &+ 32&(1-\gamma_2) &+ 160 &\epsilon &< 3
			\end{aligned}
			\right.
		\]
		\item if $d \in \{3, \ldots, 9\}$,
		\[
			\left\{
			\begin{aligned}
			\frac{1}{3 \cdot 2^d} (1-\gamma_1) + \bigg(1+\frac{1}{6(2^d-1)}\bigg) &(1- \gamma_2)
			+ 6 \epsilon &< \frac{1}{3 \cdot 2^d}, \\
			&(1-\gamma_2) + 4 \epsilon &< \frac{1}{4\cdot2^d},
			\end{aligned}
			\right.
		\]
		\item if $d \geq 10$,
		\begin{equation}
			\label{eq:55}
			\frac{2}{3d(d+1)^2} (1 -\gamma_1) + \bigg(1+\frac{1}{3d(d+1)}\bigg)(1-\gamma_2) 
			+ 6 \epsilon < \frac{2}{3d(d+1)^2},
		\end{equation}
	\end{enumerate}
	we have
	\[
		\sum_{p \in \bfP_N} e^{2\pi i \xi P(p)} \log (p)
		=
		\sum_{p \in \PP_N} e^{2\pi i \xi P(p)} \log (p) \psi(p) 
		+
		\calO\Big(\varphi_2(N) N^{-\epsilon}\Big).
	\]
\end{theorem}
\begin{proof}
	We treat $d \geq 10$ only since similar arguments apply to the other cases. Let us introduce the ``sawtooth''
	function $\Phi(x)=\{x\}-1/2$. Notice that, in view of Lemma \ref{lem:1} we have
	\begin{align*}
		\lfloor \varphi_1(n) \rfloor - \lfloor \varphi_1(n) - \psi(n) \rfloor
		=\psi(n)+\Phi\big(\varphi_1(n)-\psi(n) \big)-\Phi\big(\varphi_1(n)\big).
	\end{align*}
	Hence, we may write
	\[
		\sum_{p \in \bfP_N}e^{2\pi i \xi P(p)} \log(p)
		=\sum_{p \in \PP_N} e^{2\pi i \xi P(p)} \log(p) \psi(p)
		+\sum_{p \in \PP_N} e^{2\pi i \xi P(p)} \log(p) 
		\big(\Phi\big(\varphi_1(p)-\psi(p) \big)-\Phi\big(\varphi_1(p)\big)\big).
	\]
	Since
	\[
        \frac{1}{2} - \gamma_2 + 2\epsilon = (1- \gamma_2) - \frac{1}{2} + 2\epsilon < 0,
    \]
	by the prime number theorem we get
	\begin{align*}
		&\sum_{p \in \PP_N} e^{2\pi i \xi P(p)} \log(p)
        \big(\Phi\big(\varphi_1(p)-\psi(p) \big)-\Phi\big(\varphi_1(p)\big)\big) \\
		&\qquad\qquad=
		\sum_{n = 1}^N 
		e^{2\pi i \xi P(n)} \Lambda(n)
        \big(\Phi\big(\varphi_1(n)-\psi(n) \big)-\Phi\big(\varphi_1(n)\big)\big)
		+
		\calO\left(\varphi_2(N) N^{-\epsilon} \right).
	\end{align*}
	Next, we claim that
	\[
		\sum_{n = 1}^N
        e^{2\pi i \xi P(n)} \Lambda(n)
        \big(\Phi\big(\varphi_1(n)-\psi(n) \big)-\Phi\big(\varphi_1(n)\big)\big) 
		= \calO\left(\varphi_2(N) N^{-\epsilon}\right).
	\]
	To see this, let us expand $\Phi$ into its Fourier series, i.e.,
	\begin{align*}
		\Phi(x)=\sum_{0<|m|\le M} \frac{1}{2\pi i m} e^{-2\pi imx}
		+\calO\left(\min\left\{1, \frac{1}{M\|x\|}\right\}\right),
	\end{align*}
	for some $M>0$ where $\|x\| = \min\{\norm{x-n} : n \in \ZZ\}$ is the distance of $x \in \RR$ to the nearest
	integer. Next, we split the resulting sum into three parts,
	\begin{align*}
		I_1&=\sum_{0<|m|\le M} \frac{1}{2\pi i m} 
		\sum_{n=1}^Ne^{2\pi i(\xi P(n) -m\varphi_1(n))} \left(e^{2\pi im\psi(n)} - 1 \right) \Lambda(n),\\
	\intertext{and}
		I_2&=\mathcal O\bigg( \sum_{n=1}^N\min\left\{1, \frac{1}{M\|\varphi_1(n)-\psi(n)\|}\right\} \Lambda(n)\bigg),\\
		I_3&=\mathcal O\bigg( \sum_{n=1}^N\min\left\{1, \frac{1}{M\|\varphi_1(n)\|}\right\} \Lambda(n)\bigg).
	\end{align*}
	In this way, our aim is reduced to showing that each term $I_1, I_2$ and $I_3$ belongs to
	$\calO\big(\varphi_2(N)N^{-\epsilon}\big)$. 
	
	\noindent
	{\bf The estimate for $I_1$.}
	Let $\phi_m(x)=e^{2\pi i m\psi(x)} - 1$. Using \eqref{eq:5}, we easily see that
	\begin{equation}
		\label{eq:22}
		|\phi_m(x)| \lesssim \frac{\abs{m} \varphi_2(x)}{x},\quad\text{and}\quad
		|\phi_m'(x)|\lesssim \frac{\abs{m} \varphi_2(x)}{x^2}.
	\end{equation}
	Let us first estimate the inner sum in $I_1$. By dyadic splitting we get
	\begin{equation}
		\label{eq:34}
		\begin{aligned}
		&
		\Big|\sum_{n = 1}^N 
		\exp\Big(2\pi i\big(\xi P(n) -m\varphi_1(n)\big)\Big) \phi_m(n) \Lambda(n) \Big| \\
		&\qquad\qquad\lesssim
		(\log N)
		\max_{\atop{X < X' \leq 2X}{X' \leq N}}
		\Big|
		\sum_{X < n \leq X'}
		\exp\Big(2\pi i\big(\xi P(n) -m\varphi_1(n)\big)\Big) \phi_m(n) \Lambda(n)
		\Big|.
		\end{aligned}
	\end{equation}
	Now, by the partial summation, we have
	\[
		\Big|
        \sum_{X < n \leq X'} 
		\exp\Big(2\pi i \big(\xi P(n) -m\varphi_1(n)\big)\Big) \phi_m(n) \Lambda(n)
        \Big|
		\leq
		\abs{S(X, X')} \cdot \abs{\phi_m(X')}
		+
		\int_X^{X'}
		\abs{S(X, x)} \cdot \abs{\phi_m'(x)} {\: \rm d}x
	\]
	where
	\[
		S(X, x) = \sum_{X < n \leq x} \exp\Big(2\pi i\big(\xi P(n) -m\varphi_1(n)\big)\Big) \Lambda(n).
	\]
	It follows from Theorem \ref{thm:5}(iv) and estimates \eqref{eq:22} that
	\begin{align*}
		&
		\Big|
        \sum_{X < n \leq X'} \exp\Big(2\pi i \big(\xi P(n) -m\varphi_1(n)\big)\Big) \phi_m(n) \Lambda(n)
        \Big|\\
        &\qquad\lesssim
		\abs{m}
		\max_{X \in [1, N]} X^{\epsilon - \frac{1}{4 d(d+1)}} \varphi_2(X)
		+
		\abs{m}^{1 + \frac{1}{2d (d+1)}}
		\max_{X \in [1, N]}
		X^{\epsilon - \frac{d-1}{4d (d+1)}} \varphi_2(X) \\
		&\qquad\phantom{\lesssim}+
		\abs{m}^{1 + \frac{1}{3d(d+1)}}
		\max_{X \in [1, N]} X^{\epsilon} \varphi_2(X) 
		\big(\varphi_1(X) \sigma_1(X) \big)^{-\frac{2}{3d(d+1)^2}}\\
		&\qquad\lesssim
		\abs{m} N^{\epsilon} \varphi_2(N) \Big(
		N^{-\frac{1}{4 d(d+1)}}
		+
		\abs{m}^{\frac{1}{2d(d+1)}}
		N^{-\frac{d-1}{4d(d+1)}}
		+
		\abs{m}^{\frac{1}{3d(d+1)}}
		\big(\varphi_1(N) \sigma_1(N) \big)^{-\frac{2}{3 d (d+1)^2}} \Big),
	\end{align*}
	and hence by \eqref{eq:34}, for each $\epsilon > 0$,
	\begin{align*}
		&
		\frac{1}{\abs{m}}
		\Big|\sum_{n = 1}^N \exp\Big(2\pi i\big(\xi P(n) -m\varphi_1(n)\big)\Big) \phi_m(n) \Lambda(n) \Big|\\
		&\qquad\qquad\lesssim
		N^{\epsilon} \varphi_2(N) \Big(
		N^{-\frac{1}{4 d(d+1)}}
		+
		\abs{m}^{\frac{1}{2 d(d+1)}}
		N^{-\frac{d-1}{4d (d+1)}}
		+
		\abs{m}^{\frac{1}{3d(d+1)}}
		\big(\varphi_1(N) \sigma_1(N) \big)^{-\frac{2}{3d(d+1)^2}} \Big).
	\end{align*}
	Now, by summing up over $m \in \{1, \ldots, M \}$ we arrive at the conclusion that
	\begin{equation}
		\label{eq:25}
		\abs{I_1} \lesssim 
		M N^{\epsilon} \varphi_2(N) \Big(
		N^{-\frac{1}{4 d(d+1)}}
		+
		M^{\frac{1}{2 d (d+1)}}
		N^{-\frac{d-1}{4d(d+1)}}
		+
		M^{\frac{1}{3d(d+1)}}
		\big(\varphi_1(N) \sigma_1(N) \big)^{-\frac{2}{3d(d+1)^2}}\Big).
	\end{equation}

	\noindent
	{\bf The estimates for $I_2$ and $I_3$.}
	Let us consider $I_2$. Since (see \cite[Section 2]{hb})
	\begin{align}
		\label{eq:15}
		\min\left\{1, \frac{1}{M\|x\|}\right\}=\sum_{m\in\ZZ}c_m e^{2\pi imx}
	\end{align}
	where
	\begin{align}
	\label{eq:16}
		|c_m|\lesssim \min\left\{\frac{\log M}{M}, \frac{1}{|m|}, \frac{M}{|m|^2}\right\},
	\end{align}
	we have
	\begin{align*}
		&\sum_{n=1}^N
        \min\left\{1, \frac{1}{M\|\varphi_1(n)-\psi(n)\|}\right\} \Lambda(n)
        \leq
		(\log N)
		\sum_{m \in \ZZ}
		\abs{c_m}
		\Big|\sum_{n=1}^N e^{2\pi i m(\varphi_1(n)-\psi(n))} \Big| \\
		&\qquad\qquad\lesssim
		\frac{\log M}{M} N (\log N) + 
		(\log N)
		\bigg(\sum_{0 < \abs{m} < M} \frac{1}{\abs{m}} + \sum_{\abs{m} > M} \frac{M}{\abs{m}^2}\bigg)
		\Big|\sum_{n=1}^N e^{2\pi i m(\varphi_1(n)-\psi(n))} \Big|.
	\end{align*}
	By Proposition \ref{prop:2}(i), we get
	\begin{align*}
		\Big|\sum_{n=1}^N e^{2\pi i m(\varphi_1(n)-\psi(n))} \Big|
		&\lesssim
		\abs{m}^{\frac{1}{2}} 
		\sup_{X \in [1, N]}
		X^{1+\frac{1}{2} \epsilon} \big(\varphi_1(X) \sigma_1(X)\big)^{-\frac{1}{2}} \\
		&\lesssim
		\abs{m}^{\frac{1}{2}} N^{1+\frac{1}{2} \epsilon} \big(\varphi_1(N) \sigma_1(N)\big)^{-\frac{1}{2}},
	\end{align*}
	thus
	\begin{equation}
		\label{eq:24}
		\abs{I_2} \lesssim
		M^{-1} (\log M) N^{1+\frac{1}{2} \epsilon}
		+ M^{\frac{1}{2}} N^{1+\epsilon} \big(\varphi_1(N) \sigma_1(N)\big)^{-\frac{1}{2}}.
	\end{equation}
	Arguments similar to the above leads to the same bounds for $I_3$.
	
	\noindent
	{\bf Conclusion.}
	From estimates \eqref{eq:25} and \eqref{eq:24}, we conclude that
	\begin{align*}
		|I_1|+|I_2|+|I_3|
		&\lesssim
		M N^{\epsilon} \varphi_2(N) \Big(
        N^{-\frac{1}{4 d(d+1)}}
        +
        M^{\frac{1}{2d(d+1)}}
        N^{-\frac{d-1}{4d(d+1)}}
        +
        M^{\frac{1}{3d(d+1)}}
        \big(\varphi_1(N) \sigma_1(N) \big)^{-\frac{2}{3d(d+1)^2}} \Big)\\	
		&\phantom{\lesssim}
		+
		M^{-1} (\log M) N^{1+\frac{1}{2}\epsilon}
		+
		M^\frac{1}{2} N^{1+\epsilon} \big(\varphi_1(N) \sigma_1(N)\big)^{-\frac{1}{2}}.
	\end{align*}
	Take $M = N^{1+2\epsilon} \varphi_2(N)^{-1}$. As it may be easily verified, if $\epsilon$
	satisfies \eqref{eq:55} then
	\[
		\frac{1}{2} (1 + 2 \epsilon - \gamma_2) + 1 + \epsilon- \frac{1}{2}\gamma_1 + \frac{1}{2} \epsilon 
		- \gamma_2 + \epsilon 
		\leq
		\frac{3}{2} (1 - \gamma_2) + \frac{1}{2}(1-\gamma_1) - \frac{1}{2} + 4 \epsilon \leq 0,
	\]
	thus
	\[
		M^\frac{1}{2} N^{1+\epsilon} \big(\varphi_1(N) \sigma_1(N)\big)^{-\frac{1}{2}}
		=
		\calO\left(\varphi_2(N) N^{-\epsilon}\right).
	\]
	Since
	\[
        \bigg(\frac{2}{3d(d+1)^2} - 6\epsilon \bigg) \frac{3d(d+1)}{3d(d+1)+1} < \frac{1}{4d(d+1)} - 5 \epsilon
        < \bigg(\frac{d-1}{4d(d+1)} - 5 \epsilon \bigg)
        \frac{2d(d+1)}{2d(d+1)+1},
    \]
	for the other terms, we obtain
	\begin{align*}
		(1 + 2 \epsilon - \gamma_2) + \epsilon -\frac{1}{4 d(d+1)} + \epsilon
		\leq
		(1- \gamma_2) -\frac{1}{4 d(d+1)} + 5 \epsilon \leq 0,
	\end{align*}
	and
	\begin{align*}
		&
		\bigg(1+\frac{1}{2d(d+1)}\bigg)(1+2\epsilon - \gamma_2)
		+\epsilon
		-\frac{d-1}{4d(d+1)}
		+\epsilon \\
		&\qquad\qquad\leq
		\bigg(1+\frac{1}{2d(d+1)}\bigg) (1 - \gamma_2)
		-\frac{d-1}{4d(d+1)}
		+ 5 \epsilon \leq 0.
	\end{align*}
	Finally,	
	\begin{align*}
		&
        \bigg(1+\frac{1}{3d(d+1)}\bigg)(1+2\epsilon - \gamma_2)
		+\epsilon
		-\frac{2}{3d(d+1)^2}\gamma_1
		+\frac{2}{3d(d+1)^2}\epsilon
		+\epsilon \\
		&\qquad\qquad\leq
		\bigg(1+\frac{1}{3d(d+1)}\bigg)(1-\gamma_2) +
		\frac{2}{3d(d+1)^2} (1 - \gamma_1)
		-
		\frac{2}{3d(d+1)^2}
		+ 6 \epsilon \leq 0.
	\end{align*}
	Consequently,
	\[
        \abs{I_1}+\abs{I_2}+\abs{I_3} \lesssim \varphi_2(N) N^{-\epsilon},
	\]
	which completes the proof.
\end{proof}
 
\begin{theorem}
	\label{thm:4}
	For each $\epsilon > 0$, satisfying
	\begin{enumerate}[itemindent=*,leftmargin=0pt]
		\item if $d = 1$,
		\[
			\left\{
			\begin{aligned}
			&(1-\gamma_1) &+ 15&(1-\gamma_2) &+ 164 &\epsilon &< 1,\\
			3 &(1-\gamma_1) &+ 12&(1-\gamma_2) &+ 120 &\epsilon &< 2
			\end{aligned}
			\right.
		\]
		\item if $d = 2$,
		\[
			\left\{
			\begin{aligned}
			3&(1-\gamma_1) &+ 52&(1-\gamma_2) &+ 720&\epsilon &< 3, \\
			4&(1-\gamma_1) &+ 32&(1-\gamma_2) &+ 320&\epsilon &< 3
			\end{aligned}
			\right.
		\]
		\item if $d \in \{3, \ldots, 9\}$,
		\[
			\left\{
			\begin{aligned}
			\frac{1}{3 \cdot 2^d} (1-\gamma_1) + \bigg(1+\frac{1}{6(2^d-1)}\bigg) &(1- \gamma_2)
			+ 12 \epsilon &< \frac{1}{3 \cdot 2^d}, \\
			&(1-\gamma_2) + 8 \epsilon &< \frac{1}{4\cdot 2^d},
			\end{aligned}
			\right.
		\]
		\item if $d \geq 10$,
		\[
			\frac{2}{3d(d+1)^2} (1 -\gamma_1) + \bigg(1+\frac{1}{3d(d+1)}\bigg)(1-\gamma_2) 
			+ 12 \epsilon < \frac{2}{3d(d+1)^2},
		\]
	\end{enumerate}
	we have
	\begin{equation}
		\label{eq:6}
		\sum_{p \in \bfP_N} e^{2\pi i \xi P(p)} \frac{\log (p)}{p \psi(p)}
		=
		\sum_{p \in \PP_N} e^{2 \pi i \xi P(p)} \frac{\log (p) }{p} + \calO\big(N^{-\epsilon}\big),
	\end{equation}
	and
	\begin{equation}
		\label{eq:19}
		\sum_{p \in \bfP_N} e^{2\pi i \xi P(p)} \frac{\log (p)}{\psi(p)}
		=
		\sum_{p \in \PP_N} e^{2\pi i \xi P(p)} \log(p)
		+
		\calO\big(N^{1-\epsilon}\big).
	\end{equation}
\end{theorem}
\begin{proof}
	Set
	\[
		S(N) = \sum_{p \in \bfP_N} e^{2 \pi i \xi P(p)} \log(p),\qquad\text{and}\qquad
		U(N) = \sum_{p \in \PP_N} e^{2 \pi i \xi P(p)} \log(p)\psi(p).
	\]
	In view of Theorem \ref{thm:3},
	\begin{equation}
		\label{eq:33}
		\big| S(N) - U(N) \big|
		\lesssim
		C \varphi_2(N) N^{-2 \epsilon}.
	\end{equation}
	Notice that by the partial summation we have
	\begin{align}
		\nonumber
		\sum_{p \in \bfP_N} e^{2 \pi i \xi P(p)} \frac{\log(p)}{p \psi(p)}
		&=
		\sum_{n=2}^N \frac{1}{n \psi(n)} (S(n) - S(n-1)) \\
		\label{eq:30}
		&=
		\frac{S(N)}{N\psi(N)}
		+ \sum_{n = 2}^{N-1} \bigg(\frac{1}{n \psi(n)} - \frac{1}{(n+1) \psi(n+1)} \bigg) S(n).
	\end{align}
	Similarly, we get
	\begin{align}
		\nonumber
		\sum_{p \in \PP_N} e^{2\pi i \xi P(p)} \frac{\log(p)}{p} 
		&= \sum_{n = 2}^N \frac{1}{n \psi(n)} \big(U(n) - U(n-1)\big) \\
		\label{eq:29}
		&=
		\frac{U(N)}{N\psi(N)}
		+ \sum_{n = 2}^{N-1} \bigg(\frac{1}{n \psi(n)} - \frac{1}{(n+1) \psi(n+1)}\bigg) U(n).
	\end{align}
	Therefore, by subtracting \eqref{eq:29} from \eqref{eq:30}, we arrive at the conclusion that
	\begin{align*}
		&
		\bigg|
		\sum_{p \in \bfP_N} e^{2\pi i \xi P(p)} \frac{\log(p)}{p \psi(p)} - 
		\sum_{p \in \PP_N} e^{2 \pi i \xi P(p)} \frac{\log (p)}{p}
		\bigg| \\
		&\qquad\qquad\lesssim
		\frac{\abs{S(N) - U(N)}}{N\psi(N)} 
		+
		\sum_{n = 2}^{N-1}\bigg|\frac{1}{n\psi(n)} - \frac{1}{(n+1)\psi(n+1)}\bigg| 
		\cdot \abs{S(n) - U(n)}.
	\end{align*}
	Since, by \eqref{eq:5} and \eqref{eq:28},
	\[
		\frac{1}{N \psi(N)} \lesssim \frac{1}{\varphi_2(N) \sigma_2(N)},
	\]
	and
	\begin{align*}
		\bigg|\frac{1}{(n+1)\psi(n+1)} - \frac{1}{n \psi(n)} \bigg|
		&\leq
		\sup_{x \in [n, n+1]} \bigg|\frac{1}{x^2 \psi(x)} + \frac{\psi'(x)}{x \psi(x)^2}\bigg| \\
		&\lesssim
		\frac{1}{n \varphi_2(n) \sigma_2(n)},
	\end{align*}
	the estimate \eqref{eq:33} gives
	\[
		\frac{\abs{S(N) - U(N)}}{N\psi(N)}
		\lesssim
		N^{-\epsilon},
	\]
	and
	\[
		\bigg|\frac{1}{n\psi(n)} - \frac{1}{(n+1)\psi(n+1)}\bigg|
        \cdot \abs{S(n) - U(n)}
		\lesssim
		n^{-1-\epsilon}.
	\]
	Hence,
	\[
		\bigg|
        \sum_{p \in \bfP_N} e^{2\pi i \xi P(p)} \frac{\log(p)}{p \psi(p)} -
        \sum_{p \in \PP_N} e^{2 \pi i \xi P(p)} \frac{\log(p)}{p}
        \bigg|
		\lesssim N^{-\epsilon} + \sum_{n = 2}^{N-1} n^{-1-\epsilon} \lesssim N^{-\epsilon},
	\]
	which concludes the proof of \eqref{eq:6}. Similar considerations apply to \eqref{eq:19}.
\end{proof}

The following theorem generalizes the results obtained in \cite{leit, psch} and \cite{m3}.
\begin{theorem}
	\label{thm:7}
	\[
		\abs{\bfP_N} = \bigg( \int_2^N \frac{\psi(x)}{\log(x)} {\: \rm d} x \bigg)
		\big(1+o(1)\big).
	\]
\end{theorem}
\begin{proof}
	Set
	\[
		W_N = \sum_{p \in \bfP_N} \log(p), \qquad\text{and}\qquad
		V_N = \sum_{p \in \PP_N} \log(p) \psi(p).
	\]
	Let $\epsilon$ satisfy hypotheses of Theorem \ref{thm:3}, then
	\begin{equation}
		\label{eq:41}
		W_N = V_N + \calO\big(\varphi_2(N) N^{-\epsilon}\big).
	\end{equation}
	By the partial summation we have
	\begin{align*}
		\abs{\bfP_N} &= \sum_{n = 2}^N \big(W_n - W_{n-1} \big)\frac{1}{\log (n)} \\ 
		&= W_N \frac{1}{\log N} + \sum_{n = 2}^{N-1} W_n \bigg(\frac{1}{\log(n)} - \frac{1}{\log (n+1)}\bigg),
	\end{align*}
	and
	\[
		\sum_{p \in \PP_N} \psi(p)
		=
		V_N \frac{1}{\log N} + \sum_{n = 2}^{N-1} V_n \bigg(\frac{1}{\log(n)} - \frac{1}{\log (n+1)}\bigg).
	\]
	Therefore, by \eqref{eq:41}, we obtain
	\begin{align*}
		\Big|
		\abs{\bfP_N} - \sum_{p \in \PP_N} \psi(p)
		\Big|
		&\leq
		\big|W_N - V_N\big| \frac{1}{\log N} + \sum_{n = 2}^{N-1} \big|W_n - V_n\big| 
		\bigg(\frac{1}{\log(n)} - \frac{1}{\log(n+1)}\bigg)\\
		&\lesssim
		\varphi_2(N) N^{-\epsilon} + \sum_{n=2}^{N-1} \varphi_2(n) n^{-1-\epsilon}, 
	\end{align*}
	and thus
	\[
		\abs{\bfP_N} = \sum_{p \in \PP_N} \psi(p) + \calO\big(\varphi_2(N) N^{-\epsilon}\big).
	\]
	Setting
	\[
		\vartheta(N) = \sum_{p \in \PP_N} \log (p),
	\]
	by the summation by parts, we obtain
	\begin{align*}
		\sum_{p \in \PP_N} \psi(p) 
		&= \sum_{n = 2}^N \big(\vartheta(n) - \vartheta(n-1)\big) \frac{\psi(n)}{\log(n)} \\
		&= \vartheta(N) \frac{\psi(N)}{\log N} 
		- \vartheta(2) \frac{2}{\log 2}
		+ \sum_{n = 2}^{N-1} \vartheta(n) 
		\bigg(\frac{\psi(n)}{\log n} - \frac{\psi(n+1)}{\log(n+1)} \bigg),
	\end{align*}
	and 
	\begin{align*}
		\sum_{n = 2}^N  \frac{\psi(n)}{\log n}
        = N \frac{\psi(N)}{\log N} 
		- 2 \frac{\psi(2)}{\log 2}
		+ \sum_{n = 2}^{N-1} n 
		\bigg(\frac{\psi(n)}{\log n} - \frac{\psi(n+1)}{\log(n+1)} \bigg).
	\end{align*}
	The prime number theorem implies that
	\begin{equation}
		\label{eq:42}
		\vartheta(N) = N\Big(1+\calO\big(N^{-2 \epsilon}\big)\Big).
	\end{equation}
	Moreover, by \eqref{eq:5} and \eqref{eq:28},
	\begin{align*}
		\bigg|\frac{\psi(n)}{\log(n)} - \frac{\psi(n+1)}{\log(n+1)}\bigg|
		&\leq
		\sup_{x \in [n, n+1]}
		\bigg|\frac{\psi'(x) \log(x) - \psi(x) x^{-1}}{(\log(n))^2}\bigg| \\
		&\lesssim
		\varphi_2(n) n^{-2 +\epsilon}.
	\end{align*}
	Hence,
	\begin{align*}
		\bigg|
		\sum_{p \in \PP_N} \psi(p) - \sum_{n = 2}^N  \frac{\psi(n)}{\log n}
		\bigg|
		&\leq
		\big|\vartheta(N) - N\big|\frac{\psi(N)}{\log N} + 
		\sum_{n = 2}^{N-1} \big|\vartheta(n) - n\big| \bigg|\frac{\psi(n)}{\log(n)} - \frac{\psi(n+1)}{\log(n+1)}\bigg|\\
		&\lesssim
		\varphi_2(N) N^{-\epsilon}.
	\end{align*}
	Finally, 
	\begin{align*}
		\bigg|
		\sum_{n = 2}^{N-1} \frac{\psi(n)}{\log(n)} - \int_2^N \frac{\psi(x)}{\log(x)} {\: \rm d}x \bigg|
		&\lesssim
		\sum_{n = 2}^{N-1} \int_0^1 \bigg|\frac{\psi(n)}{\log n} - \frac{\psi(n+t)}{\log(n+t)}\bigg| {\: \rm d}t \\
		&\lesssim
		\sum_{n = 2}^{N-1} \vphi_2(n) n ^{-2 + \epsilon}
	\end{align*}
	Thus
	\[
		\sum_{n = 2}^N \frac{\psi(n)}{\log(n)} = \int_2^N \frac{\psi(x)}{\log(x)} {\: \rm d}x
		+\calO\big(\vphi_2(N) N^{-1+\epsilon}\big)
	\]
	Now, using \eqref{eq:5}, we get
	\begin{align*}
		\int_2^N \frac{\psi(x)}{\log(x)} {\: \rm d}x 
		&\geq
		\frac{1}{\log N} \int_2^N \psi(x) {\: \rm d}x \\
		&\gtrsim
		\frac{1}{\log N} \int_2^N \vphi_2'(x) {\: \rm d}x\\
		&\gtrsim
		\frac{\varphi_2(N)}{\log(N)},
	\end{align*}
	which completes the proof.
\end{proof}

\section{Variational estimates}
To deal with $r$-variational estimates for averaging operators and truncated discrete Hilbert transform, we apply the
method used in \cite{zk} and \cite[Section 4]{mtz}. For $\rho \in (0, 1)$ we set
$Z_\rho = \big\{\lfloor 2^{k^\rho} \rfloor : k \in \NN \big\}$ and 
define long $r$-variations by
\[
	V_r^L(a_n : n \in \NN) = V_r(a_n : n \in Z_\rho).
\]
Then the corresponding short variations are given by
\[
	V_r^S(a_n : n \in \NN) = \Big(\sum_{k \geq 1} V_r\big(a_n : n \in [N_{k-1}, N_k)\big)^r \Big)^{\frac{1}{r}},
\]
where $N_k = \lfloor 2^{k^\rho}\rfloor$. Observe that
\[
	V_r(a_n : n \in \NN) \lesssim V_r^L(a_n : n \in \NN) + V_r^S(a_n : n \in \NN).
\]
\subsection{Averaging operators}
In this section we prove Theorem \ref{thm:1} for the model dynamical system. 

Given a function $f$ on $\ZZ$ we set
\[
	\calA_N f(x) = \frac{1}{\abs{\bfP_N}} \sum_{p \in \bfP_N} f\big(x - P(p)\big).
\]
While studying $r$-variations we may replace the operators $\calA_N$ by the weighted averages $\calM_N$,
\[
	\calM_N f(x) = \frac{1}{\Psi_N} \sum_{p \in \bfP_N} f\big(x - P(p) \big) \frac{\log p}{\psi(p)},
\]
where
\[
	\Psi_N = \sum_{p \in \bfP_N} \frac{\log(p)}{\psi(p)}.
\]
Indeed, since $\psi$ is decreasing the ratio of weights in $\calA_N$ and $\calM_N$ is monotonically decreasing, thus by 
\cite[Proposition 5.2]{mtz}, there is $C > 0$ such that for all $r > 2$,
\[
	V_r(\calA_N f(x) : N \in \NN) \leq C \cdot V_r(\calM_N f(x) : N \in \NN),
\]
where the constant $C$ is independent of $f$, $x$ and $r$. Therefore, it is enough to show the following theorem.
\begin{theorem}
	\label{thm:6}
	For each $s > 1$ there is $C_s > 0$ such that for all $r > 2$ and $f \in \ell^s(\ZZ)$,
	\[
		\big\|
		V_r(\calM_N f : N \in \NN)
		\big\|_{\ell^s}
		\leq
		C_s \frac{r}{r-2}
		\|f\|_{\ell^s}.
	\]
\end{theorem}
\begin{proof}
	We start with short variations. Let us denote by $m_n$ the convolution kernel corresponding to $\calM_n$. Then
	for each $x \in \bfP_{N_{k-1}}$,
	\[
		\sum_{n = N_{k-1}}^{N_k-1} \abs{m_{n+1}(x) - m_n(x)}
		=
		\big(\Psi_{N_{k-1}}^{-1} - \Psi_{N_k}^{-1}\big) \frac{\log x}{\psi(x)}.
	\]
	On the other hand, for $x \in \bfP_{N_k} \setminus \bfP_{N_{k-1}}$,
	\[
		\sum_{n = N_{k-1}}^{N_k-1} \abs{m_{n+1}(x) - m_n(x)}
		\leq
		2 \Psi_{N_{k-1}}^{-1} \frac{\log x}{\psi(x)}.
	\]
	Therefore,
	\[
		\Big\|
        \sum_{n = N_{k-1}}^{N_k-1} \abs{m_{n+1}-m_n}
        \Big\|_{\ell^1}
		\leq
		\big(\Psi_{N_{k-1}}^{-1} - \Psi_{N_k}^{-1}\big) \Psi_{N_{k-1}}
		+
		2 \Psi_{N_{k-1}}^{-1} \big( \Psi_{N_{k-1}} - \Psi_{N_k} \big).
	\]
	Let $\epsilon > 0$ satisfy the hypotheses of Theorem \ref{thm:4}. By \eqref{eq:42}
	and \eqref{eq:19}, we get
	\[
		\Psi_N = \vartheta(N) + \calO\big(N^{1-\epsilon}\big) = N + \calO\big(N^{1-\epsilon}\big),
	\]
	and thus
	\begin{align*}
		\Psi_{N_k} - \Psi_{N_{k-1}} 
		\lesssim N_k - N_{k-1} + N_{k-1}^{1-\epsilon}
		\lesssim k^{\rho-1} N_{k-1} \lesssim k^{\rho-1} \Psi_{N_{k-1}}.
	\end{align*}
	Therefore, by Young's inequality,
	\begin{align}
		\nonumber
		\Big\|
		\sum_{n = N_{k-1}}^{N_k-1} \big|\calM_{n+1} f - \calM_n f \big|
		\Big\|_{\ell^s}
		&\leq
		\Big\|
		\sum_{n = N_{k-1}}^{N_k-1} \abs{m_{n+1}-m_n}
		\Big\|_{\ell^1}
		\|f\|_{\ell^s} \\
		\label{eq:23}
		&\lesssim
		k^{\rho-1} \|f\|_{\ell^s}.
	\end{align}
	Let $u = \min\{2, s\}$. By monotonicity and Minkowski's inequality, we get
	\begin{align*}
		\Big\|
		\Big(\sum_{k \geq 1} V_r\big(\calM_n f : n \in [N_{k-1}, N_k) \big)^r
		\Big)^{\frac{1}{r}}
		\Big\|_{\ell^s}
		\leq
		\Big\|
		\Big(
		\sum_{k \geq 1}
		\Big(
		\sum_{n = N_{k-1}}^{N_k-1} \big|\calM_{n+1} f - \calM_n f \big|
		\Big)^u
		\Big)^{\frac{1}{u}}
		\Big\|_{\ell^s} \\
		\leq
		\Big(\sum_{k \geq 1} 
		\Big\|
		\sum_{n = N_{k-1}}^{N_k-1} \big|\calM_{n+1} f - \calM_n f \big|
		\Big\|_{\ell^s}^u
		\Big)^{\frac{1}{u}},
	\end{align*}
	which together with \eqref{eq:23} gives
	\[
		\Big\|
        \Big(\sum_{k \geq 1} V_r\big(\calM_n f : n \in [N_{k-1}, N_k) \big)^r
        \Big)^{\frac{1}{r}}
        \Big\|_{\ell^s}
        \lesssim
		\Big(\sum_{k \geq 1} k^{-u(1-\rho)} \Big)^{\frac{1}{u}} \|f\|_{\ell^s}.
	\]
	We notice that the last sum is finite whenever $0 < \rho < \frac{u-1}{u}$.
	
	To control long $r$-variations over the set $Z_\rho$, for any $\rho \in (0, 1)$, we replace
	$\calM_N$ by a weighted average over prime numbers
	\[
		M_N f(x) = \frac{1}{\vartheta(N)} \sum_{p \in \PP_N} f\big(x - P(p)\big) \log(p).
	\]
	Since both $\calM_N$ and $M_N$ are averaging operators, we have
	\begin{equation}
		\label{eq:17}
		\big\|\calM_N f - M_N f\big\|_{\ell^s} 
		\leq \big\|\calM_N f\|_{\ell^s} + \big\|M_N f\big\|_{\ell^s}
		\leq 2 \|f\|_{\ell^s}.
	\end{equation}
	On the other hand, by Plancherel's Theorem 
	\[
		\big\|\calM_N f - M_N f\big\|_{\ell^2} \leq
		\sup_{\xi \in [0, 1]} \bigg| 
		\frac{1}{\Psi_N} \sum_{p \in \bfP_N} e^{2\pi i \xi P(p)} \frac{\log(p)}{\psi(p)}
		- \frac{1}{\vartheta(N)} \sum_{p \in \PP_N} e^{2\pi i \xi P(p)} \log(p) \bigg| \cdot \|f\|_{\ell^2},
	\]
	which together with \eqref{eq:19} and \eqref{eq:42}, implies that there is $\delta > 0$, such that
	\begin{equation}
		\label{eq:18}
		\big\|\calM_N f - M_N f\big\|_{\ell^2} \lesssim N^{-\delta} \|f\|_{\ell^2}.
	\end{equation}
	Now, interpolating between \eqref{eq:17} and \eqref{eq:18}, one can find $\delta_s > 0$ such that
	\[
		\big\|\calM_N f - M_N f\big\|_{\ell^s} \lesssim N^{-\delta_s} \|f\|_{\ell^s}.
	\]
	Hence,
	\begin{align*}
		\big\|V_r\big(\calM_N f - M_N f : N \in Z_\rho \big)\big\|_{\ell^s}
		&\lesssim
		\sum_{N \in Z_\rho} \big\|\calM_N f - M_N f \big\|_{\ell^s} \\
		&\lesssim
		\Big(\sum_{N \in Z_\rho} N^{-\delta_s}\Big)
		\|f \|_{\ell^s},
	\end{align*}
	which is bounded. Finally, by \cite[Theorem C]{tr1},
	\[
		\big\|
		V_r\big( M_N f : N \in \NN \big)
		\big\|_{\ell^s}
		\lesssim
		\frac{r}{r-2} \|f\|_{\ell^s},
	\]
	and the theorem follows.
\end{proof}

\subsection{Variational Hilbert transform}
In this section we show Theorem \ref{thm:2} for the model dynamical system and the truncated discrete Hilbert transform
defined as
\[
	\calH_N f(x) = \sum_{p \in \pm \bfP_N} f\big(x - P(p)\big) \frac{\log (\abs{p})}{p \psi(\abs{p})}.
\]
\begin{theorem}
	For each $s > 1$ there is $C_s > 0$ such that for all $r > 2$, and $f \in \ell^s(\ZZ)$,
	\[
		\big\|
		V_r(\calH_n f : n \in \NN)
		\big\|_{\ell^s}
		\leq
		C_s \frac{r}{r-2}
		\|f \|_{\ell^s}.
	\]  
\end{theorem}
\begin{proof}
	Let $h_n$ denote the convolution kernel corresponding to $\calH_n$. Then for each
	$x \in \bfP_{N_k} \setminus \bfP_{N_{k-1}}$,
	\[
		\sum_{n = N_{k-1}}^{N_k-1} \abs{h_{n+1}(x) - h_n(x)} \leq \frac{\log(x)}{x \psi(x)},
	\]
	otherwise the sum equals zero. Let us recall that the Mertens theorem says (see \cite[\S 55]{landau})
	\begin{equation}
		\label{eq:27}
		\sum_{p \in \PP_N} \frac{\log(p)}{p} = \log(N) - B_3 + \calO\Big(\exp\big(-\sqrt[14]{\log(N)}\big)\Big),
	\end{equation}
	where $B_3$ is the Mertens constant. Hence, by taking in \eqref{eq:6}, $\xi = 0$, we get
	\begin{align*}
		\sum_{p \in \bfP_{N_k} \setminus \bfP_{N_{k-1}}} \frac{\log(p)}{p\psi(p)}
		&=
		\sum_{p \in \PP_{N_k} \setminus \PP_{N_{k-1}}} \frac{\log(p)}{p} 
		+ \calO\big(N_{k-1}^{-\delta}\big) \\
		&=
		\log N_k - \log N_{k-1} + \calO\big(N_{k-1}^{-\delta}\big).
	\end{align*}
	Therefore, by the mean value theorem,
	\[
		\sum_{p \in \bfP_{N_k} \setminus \bfP_{N_{k-1}}} \frac{\log(p)}{p\psi(p)}
		\lesssim
		k^{-1 + \rho},
	\]
	and hence, we can estimate
	\[
		\Big\|
		\sum_{n = N_{k-1}}^{N_k-1} \abs{h_n - h_{n-1}}
		\Big\|_{\ell^1}
		\lesssim
		\sum_{p \in \bfP_{N_k} \setminus \bfP_{N_{k-1}}} \frac{\log(p)}{p\psi(p)}
		\lesssim k^{\rho-1}.
	\]
	Now, by Young's inequality, we conclude that
	\begin{align}
		\nonumber
		\Big\|
		\sum_{n = N_{k-1}}^{N_k-1} 
		\big|
		\calH_{n+1} f - \calH_n f
		\big|
		\Big\|_{\ell^s}
		&\leq
		\Big\|
		\sum_{n = N_{k-1}}^{N_k-1} \abs{h_{n+1} - h_n}
		\Big\|_{\ell^1} \cdot \|f\|_{\ell^s} \\
		\label{eq:40}
		&\lesssim
		k^{-1+\rho} \|f\|_{\ell^s}.
	\end{align}
	Taking $u = \min\{2, s\}$, by monotonicity and Minkowski's inequality, we get
	\begin{align*}
		\Big\|
        \Big(\sum_{k \geq 1} V_r\big(\calH_n f : n \in [N_{k-1}, N_k) \big)^r
        \Big)^{\frac{1}{r}}
        \Big\|_{\ell^s}
        \leq
        \Big\|
        \Big(
        \sum_{k \geq 1}
        \Big(
        \sum_{n = N_{k-1}}^{N_k-1} \big|\calH_{n+1} f - \calH_n f \big|
        \Big)^u
        \Big)^{\frac{1}{u}}
        \Big\|_{\ell^s} \\
        \leq
        \Big(\sum_{k \geq 1}
        \Big\|
        \sum_{n = N_{k-1}}^{N_k-1} \big|\calH_{n+1} f - \calH_n f \big|
        \Big\|_{\ell^s}^u
        \Big)^{\frac{1}{u}},
	\end{align*}
	which together with \eqref{eq:40}, for $0 < \rho < \frac{u-1}{u}$, entails that
	\[
		\Big\|
        \Big(\sum_{k \geq 1} V_r\big(\calH_n f : n \in [N_{k-1}, N_k) \big)^r
        \Big)^{\frac{1}{r}}
        \Big\|_{\ell^s}
		\lesssim
		\Big(\sum_{k \geq 1} k^{-(1-\rho) u} \Big)^{\frac{1}{u}} \|f\|_{\ell^s} \lesssim \|f\|_{\ell^s}.
	\]

	Let us now turn to estimating the long $r$-variations. Let $\epsilon > 0$ satisfy the hypotheses of Theorem
	\ref{thm:4}. We are going to replace the operators $\calH_N$, by
	\[
		H_N f(x) = \sum_{p \in \pm \PP_N} f\big(x - P(p)\big) \frac{\log (\abs{p})}{p}.
	\]
	To do so, let us observe that Theorem \ref{thm:4} implies that
	\[
		\sum_{p \in \bfP_N} \frac{\log(p)}{p \psi(p)}
		\lesssim
		\sum_{p \in \PP_N} \frac{\log(p)}{p} + N^{-\epsilon}
		\lesssim
		\log N,
	\]
	where the last estimate follows from \eqref{eq:27}. Hence, by Young's inequality we obtain 
	\begin{equation}
		\label{eq:31}
		\big\| \calH_N f - H_N f \big\|_{\ell^s} \lesssim 
		\bigg( \sum_{p \in \bfP_N} \frac{\log(p)}{p \psi(p)} +
        \sum_{p \in \PP_N} \frac{\log(p)}{p} \bigg) \cdot
		\|f\|_{\ell^s}
		\lesssim (\log N) \|f\|_{\ell^s}.
	\end{equation}
	For $s = 2$, by the Plancherel's theorem and Theorem \ref{thm:4},
	\begin{align}
		\nonumber
		\big\| \calH_N f - H_N f \big\|_{\ell^2}
		&\leq
		\sup_{\xi \in [0, 1]}
		\bigg|
		\sum_{p \in \pm \bfP_N} e^{2 \pi i \xi P(p)} \frac{\log(\abs{p})}{p \psi(\abs{p})}
		-
		\sum_{p \in \pm \PP_N} e^{2 \pi i \xi P(p)} \frac{\log(\abs{p})}{p}
		\bigg|
		\cdot
		\|f\|_{\ell^2} \\
		\label{eq:32}
		&\lesssim
		N^{-\epsilon} \|f\|_{\ell^2}.
	\end{align}
	Hence, by interpolation between \eqref{eq:31} and \eqref{eq:32}, we obtain
	\[
		\big\| \calH_N f - H_N f \big\|_{\ell^s}
		\lesssim
		N^{-\delta_s} \|f\|_{\ell^s},
	\]
	for some $\delta_s > 0$. Therefore,
	\begin{align*}
		\big\|V_r\big(\calH_N f - H_N f : N \in Z_\rho \big)\big\|_{\ell^s}
		&\leq
		\sum_{N \in Z_\rho} \big\|\calH_N f - H_N f \big\|_{\ell^s} \\
		&\lesssim
		\Big(\sum_{N \in Z_\rho} N^{-\delta_s}\Big)
		\|f \|_{\ell^s},
	\end{align*}
	which is bounded. Finally, the estimate
	\[
		\big\|V_r(H_N f(x) : N \in \NN) \big\|_{\ell^s}
		\leq
		C_s \frac{r}{r-2} \|f\|_{\ell^s}
	\]
	follows by \cite[Theorem C]{tr1}.
\end{proof}

\begin{bibliography}{discrete}
	\bibliographystyle{amsplain}

\providecommand{\bysame}{\leavevmode\hbox to3em{\hrulefill}\thinspace}
\providecommand{\MR}{\relax\ifhmode\unskip\space\fi MR }
\providecommand{\MRhref}[2]{%
  \href{http://www.ams.org/mathscinet-getitem?mr=#1}{#2}
}
\providecommand{\href}[2]{#2}
\begin{thebibliography}{10}

\bibitem{birk}
G.D. Birkhoff, \emph{Proof of the ergodic theorem}, Proc. Natl. Acad. Sci. USA
  \textbf{17} (1931), 656--660.

\bibitem{bou-p}
J.~Bourgain, \emph{An approach to pointwise ergodic theorems}, Geometric
  Aspects of Functional Analysis, Springer, 1988, pp.~204--223.

\bibitem{bou1}
\bysame, \emph{{O}n the maximal ergodic theorem for certain subsets of the
  integers}, Israel J. Math. \textbf{61} (1988), 39--72.

\bibitem{bou2}
\bysame, \emph{{O}n the pointwise ergodic theorem on {$L^p$} for arithmetic
  sets}, Israel J. Math. \textbf{61} (1988), 73--84.

\bibitem{bou}
\bysame, \emph{Pointwise ergodic theorems for arithmetic sets. {W}ith an
  appendix by the author, {H}arry {F}urstenberg, {Y}itzhak {K}atznelson and
  {D}onald {S}. {O}rnstein.}, Publ. Math.-Paris \textbf{69} (1989), no.~1,
  5--45.

\bibitem{BM}
Z.~Buczolich and R.D. Mauldin, \emph{Divergent square averages}, Ann. Math.
  \textbf{171} (2010), no.~3, 1479--1530.

\bibitem{cjrw}
J.T Campbell, R.L. Jones, K.~Reinhold, and M.~Wierdl, \emph{Oscillation and
  variation for the {H}ilbert transform}, Duke Math. J. \textbf{105} (2000),
  59--83.

\bibitem{cot}
M.~Cotlar, \emph{A unified theory of {H}ilbert transforms and ergodic
  theorems}, Rev. Mat. Cuyana \textbf{1} (1955), no.~2, 105--167.

\bibitem{GK}
W.~Graham and G.~Kolesnik, \emph{Van der {C}orput's method of exponential
  sums}, London Mathematical Society Lecture Note Series, Cambridge University
  Press, 1991.

\bibitem{hb}
D.R. Heath-Brown, \emph{The {P}jateckii--{S}hapiro prime number theorem}, J.
  Number Theory \textbf{16} (1983), 242--266.

\bibitem{hb2}
\bysame, \emph{A new {$k$}-th derivative estimate for exponential sums via
  {V}inogradov's mean value}, Tr. Mat. Inst. Steklova \textbf{296} (2017),
  95--110.

\bibitem{jkrw}
R.L. Jones, R.~Kaufman, J.M. Rosenblatt, and M.~Wierdl, \emph{Oscillation in
  ergodic theory}, Ergodic Theory Dynam. Syst. \textbf{18} (1998), no.~4,
  889--935.

\bibitem{K}
B.~Krause, \emph{Polynomial ergodic averages converge rapidly: Variations on a
  theorem of {B}ourgain}, arXiv:1402.1803, 2014.

\bibitem{kmt0}
B.~Krause, M.~Mirek, and B.~Trojan, \emph{On the {H}ardy--{L}ittlewood majorant
  problem for arithmetic sets}, J. Funct. Anal. \textbf{271} (2016), 164--181.

\bibitem{landau}
E.~Landau, \emph{Handbuch der {L}ehre von der {V}erteilung der {P}rimzahlen},
  Teubner, 1909.

\bibitem{LaV}
P.~LaVictoire, \emph{Universally {$L^1$}-bad arithmetic sequences}, J. Anal.
  Math. \textbf{113} (2011), no.~1, 241--263.

\bibitem{leit}
D.~Leitmann, \emph{The distribution of prime numbers in sequences of the form
  $[f(n)]$}, P. Lond. Math. Soc. \textbf{35} (1977), no.~3, 448--462.

\bibitem{m2}
M.~Mirek, \emph{{$\ell^p(\ZZ)$}-boundedness of discrete maximal functions along
  thin subsets of primes and pointwise ergodic theorems}, Math. Z. \textbf{279}
  (2015), no.~1--2, 27--59.

\bibitem{m3}
\bysame, \emph{Roth's theorem in the {P}iatetski-{S}hapiro primes}, Rev. Mat.
  Iberoam. \textbf{31} (2015), 617--656.

\bibitem{mst2}
M.~Mirek, E.M. Stein, and B.~Trojan, \emph{{$\ell^p\big(\ZZ^d\big)$}-estimates
  for discrete operators of {R}adon type {II}: {V}ariational estimates},
  Invent. Math. \textbf{209} (2017), no.~3, 665--748.

\bibitem{mt2}
M.~Mirek and B.~Trojan, \emph{Cotlar's ergodic theorem along the prime
  numbers}, J. Fourier Anal. Appl. \textbf{21} (2015), no.~4, 822--848.

\bibitem{mt3}
\bysame, \emph{Discrete maximal functions in higher dimensions and applications
  to ergodic theory}, Amer. J. Math. \textbf{138} (2016), no.~6, 1495--1532.

\bibitem{mtz}
M.~Mirek, B.~Trojan, and P.~Zorin-Kranich, \emph{Variational estimates for
  averages and truncated singular integrals along the prime numbers}, Trans.
  Amer. Math. Soc. \textbf{369} (2017), no.~8, 5403--5423.

\bibitem{na1}
R.~Nair, \emph{On polynomials in primes and {J}. {B}ourgain's circle method
  approach to ergodic theorems}, Ergodic Theory Dynam. Syst. \textbf{11}
  (1991), 485--499.

\bibitem{na2}
\bysame, \emph{On polynomials in primes and {J}. {B}ourgain's circle method
  approach to ergodic theorems {II}}, Stud. Math. \textbf{105} (1993), no.~3,
  207--233.

\bibitem{psch}
I.I. Pyatetskii-Shapiro, \emph{On the distribution of prime numbers in
  sequences of the form $[f(n)]$}, Mat. Sb. \textbf{33} (1953), no.~3,
  559--566.

\bibitem{titch}
E.C. Titchmarsh, \emph{The theory of the {R}iemman {Z}eta-function}, 2 ed.,
  Oxford Science Publications, 1986.

\bibitem{tr1}
B.~Trojan, \emph{Variational estimates for discrete operators modeled on
  multi-dimensional polynomial subsets of primes}, to appear in Math. Ann,
  arXiv: 1803.05406, 2018.

\bibitem{VDC}
J.G. van~der Corput, \emph{{N}eue zahlentheoretische {A}bschatzungen {II}},
  Math. Z. \textbf{29} (1929), 397--426.

\bibitem{vau0}
R.C. Vaughan, \emph{Sommes trigonom{\'e}triques sur les nombres premiers}, C.
  R. Acad. Sci. Paris \textbf{285} (1977), no.~16, 981--983.

\bibitem{wrl}
M.~Wierdl, \emph{{P}ointwise ergodic theorem along the prime numbers}, Israel
  J. Math. \textbf{64} (1988), no.~3, 315--336.

\bibitem{zk}
P.~Zorin-Kranich, \emph{Variation estimates for averages along primes and
  polynomials}, J. Funct. Anal. \textbf{268} (2015), no.~1, 210--238.

\end{thebibliography}
\end{bibliography}

\end{document}